\documentclass[12pt]{smfart}
\usepackage[francais]{babel}
\usepackage[T1]{fontenc}
\usepackage{amscd,verbatim}
\usepackage{amssymb}
\usepackage[all]{xy}
\usepackage[colorlinks,linkcolor=blue,citecolor=blue,urlcolor=red]{hyperref}

\newcommand{\Spec}{\operatorname{Spec}}
\newcommand{\Supp}{\operatorname{Supp}}
\renewcommand{\Vec}{\operatorname{\bf Vec}}
\newcommand{\Rep}{\operatorname{\bf Rep}}
\newcommand{\Ker}{\operatorname{Ker}}
\newcommand{\IM}{\operatorname{Im}}
\newcommand{\car}{\operatorname{car}}
\newcommand{\Coker}{\operatorname{Coker}}

\newcommand{\Hom}{\operatorname{Hom}}
\newcommand{\uHom}{\operatorname{\underline{Hom}}}
\newcommand{\End}{\operatorname{End}}
\newcommand{\Br}{\operatorname{Br}}

\newcommand{\Tr}{\operatorname{Tr}}
\newcommand{\Alb}{\operatorname{Alb}}
\newcommand{\Pic}{\operatorname{Pic}}
\newcommand{\NS}{\operatorname{NS}}
\newcommand{\gm}{{\operatorname{gm}}}
\newcommand{\gen}{{\operatorname{gen}}}
\newcommand{\eff}{{\operatorname{eff}}}

\newcommand{\degtr}{{\operatorname{degtr}}}
\renewcommand{\o}{{\operatorname{o}}}
\newcommand{\rat}{{\operatorname{rat}}}

\renewcommand{\hom}{{\operatorname{hom}}}
\newcommand{\num}{{\operatorname{num}}}
\newcommand{\proj}{{\operatorname{proj}}}
\newcommand{\Nis}{{\operatorname{Nis}}}
\newcommand{\Ens}{\operatorname{\bf Ens}}
\newcommand{\Sm}{\operatorname{\bf Sm}}
\newcommand{\DM}{\operatorname{\bf DM}}
\newcommand{\Cor}{\operatorname{\bf Cor}}
\newcommand{\Corf}{\operatorname{\bf Cor_f}}
\newcommand{\PST}{\operatorname{\bf PST}}
\newcommand{\Mot}{\operatorname{\bf Mot}}
\newcommand{\Chow}{\operatorname{\bf Chow}}

\newcommand{\corps}{\operatorname{\bf corps}}

\newcommand{\Cat}{\operatorname{\bf Cat}}
\newcommand{\Grp}{\operatorname{\bf Grp}}
\newcommand{\SAb}{\operatorname{\bf SAb}}
\newcommand{\Ab}{\operatorname{\bf Ab}}
\newcommand{\tore}{\operatorname{\bf tore}}
\newcommand{\res}{\operatorname{\bf res}}

\newcommand{\by}{\xrightarrow}

\newcommand{\iso}{\by{\sim}}

\newcommand{\inj}{\hookrightarrow}
\newcommand{\tto}{\dashrightarrow}
\newcommand{\surj}{\rightarrow\!\!\!\!\!\rightarrow}
\newcommand{\colim}{\varinjlim}
\renewcommand{\lim}{\varprojlim}

\DeclareFontFamily{U}{wncy}{}
\DeclareFontShape{U}{wncy}{m}{n}{%
<5>wncyr5%
<6>wncyr6%
<7>wncyr7%
<8>wncyr8%
<9>wncyr9%
<10>wncyr10%
<11>wncyr10%
<12>wncyr6%
<14>wncyr7%
<17>wncyr8%
<20>wncyr10%
<25>wncyr10}{}
\DeclareMathAlphabet{\cyr}{U}{wncy}{m}{n}

\newcommand{\Sha}{\cyr{X}}
\newcommand{\sha}{\cyr{x}}

\newcommand{\sA}{\mathcal{A}}
\newcommand{\sB}{\mathcal{B}}
\newcommand{\sC}{\mathcal{C}}
\newcommand{\sH}{\mathcal{H}}

\newcommand{\sK}{\mathcal{K}}
\newcommand{\sM}{\mathcal{M}}
\newcommand{\sN}{\mathcal{N}}
\newcommand{\sS}{\mathcal{S}}
\newcommand{\sX}{\mathcal{X}}
\newcommand{\sY}{\mathcal{Y}}
\newcommand{\sZ}{\mathcal{Z}}

\newcommand{\A}{\mathbf{A}}
\newcommand{\C}{\mathbf{C}}

\newcommand{\N}{\mathbf{N}}
\renewcommand{\P}{\mathbf{P}}
\newcommand{\Q}{\mathbf{Q}}
\newcommand{\Z}{\mathbf{Z}}

\newcommand{\bL}{\mathbb{L}}

\newcommand{\cl}{\operatorname{cl}}
\newcommand{\un}{\mathbf{1}}

\renewcommand{\epsilon}{\varepsilon}
\renewcommand{\phi}{\varphi}

\newtheorem{lemme}{Lemme}[section]
\newtheorem{prop}[lemme]{Proposition}
\newtheorem{thm}[lemme]{Th\'eor\`eme}
\newtheorem{cor}[lemme]{Corollaire}
\newtheorem{Th}{Th\'eor\`eme}
\newtheorem{Coro}{Corollaire}
\theoremstyle{definition}
\newtheorem{defn}[lemme]{D\'efinition}
\theoremstyle{remark}
\newtheorem{rque}[lemme]{Remarque}
\newtheorem{ex}[lemme]{Exemple}
\newtheorem{exs}[lemme]{Exemples}

\newenvironment{thlist}{\begin{list}{\rm{(\roman{enumi})}}%
{\usecounter{enumi}}}%
{\end{list}}

\newcommand{\subjclassname}
    {Classification math\'ematique par sujets \textup{(2010)}}%

\numberwithin{equation}{section}

\begin{document}

\title{Motifs et adjoints}
\author{Bruno Kahn}
\address{Institut de Math\'ematiques de Jussieu - PRG\\Case 247\\
4 place Jussieu\\
75252 Paris Cedex 05\\
FRANCE}
\email{bruno.kahn@imj-prg.fr}
\date{7 mars 2017}
\begin{abstract} On démontre dans de nombreux cas l'existence d'adjoints à l'extension des scalaires sur des catégories de nature motivique, dans le cadre d'extensions de corps. Cette situation est différente de celle, étudiée habituellement, d'un morphisme de type fini de schémas. On en tire diverses applications, dont une construction d'un ``motif de Tate-\v Safarevi\v c'' attaché à une variété abélienne sur un corps de fonctions. On en déduit aussi une approche possible de la conjecture de Bloch sur les surfaces, par réduction aux courbes.
\end{abstract}
\begin{altabstract} We show in many cases the existence of adjoints to extension of scalars on categories of motivic nature, in the framework of field extensions. This is to be contrasted with the more classical situation where one deals with a finite type  morphism of schemes. Among various applications, one is a construction of a ``Tate-\v Safarevi\v c motive" attached to an abelian variety over a function field. We also deduce a possible approach to \text{Bloch's} conjecture on surfaces, by reduction to curves.
\end{altabstract}
\maketitle

\enlargethispage*{60pt}

\tableofcontents

\section*{Introduction}

Le formalisme des six op\'erations, imagin\'e par Grothendieck, est au c\oe ur des techniques cohomologiques et motiviques en g\'eom\'etrie alg\'ebrique: \'etant donn\'e une cat\'egorie de sch\'emas $\sC$, il s'agit d'un syst\`eme de pseudo-foncteurs
\[f^*,f_*,f_!,f^!:\sC\to \Cat\]
o\`u $\Cat$ est la $2$-cat\'egorie des cat\'egories, les quatre foncteurs prennent la m\^eme valeur $\sM(X)$ sur un objet $X\in \sC$, et pour un morphisme $f:X\to Y$:
\begin{itemize}
\item $f^*,f^!$ envoient $\sM(Y)$ vers $\sM(X)$;
\item $f_*,f_!$ envoient $\sM(X)$ vers $\sM(Y)$;
\item $f_*$ est adjoint \`a droite de $f^*$; $f^!$ est adjoint \`a droite de $f_!$.
\end{itemize}

On a des propri\'et\'es suppl\'ementaires: th\'eor\`emes de changement de base, structures tensorielles, dualit\'e; je n'entrerai pas dans les d\'etails, renvoyant à \cite{resdu}, SGA4 et \cite{ayoub}. 

Le but de cet article est d'\'etudier une situation plus simple qui conduit \`a un nombre surprenant de r\'esultats: celle o\`u $\sC$ est une \emph{cat\'egorie de corps}. Plus pr\'ecis\'ement, soit $\sM$ un pseudo-foncteur de la cat\'egorie des corps commutatifs vers une $2$-cat\'egorie.
Si $K/k$ est une extension, on a donc un $1$-morphisme
\begin{equation}\label{eq1}
\sM(k)\to \sM(K).
\end{equation}

Les questions que je me propose d'\'etudier dans cette article sont: 

\begin{enumerate}
\item[(1)] Quand le $1$-morphisme \eqref{eq1} admet-il un adjoint (\`a gauche ou \`a droite)?
\item[(2)] Si $\sM'$ est un autre pseudo-foncteur (\`a valeurs dans la m\^eme $2$-cat\'egorie) et 
qu'on a un morphisme de pseudo-foncteurs $\phi:\sM\to \sM'$, quel est l'effet de
ce morphisme sur les adjoints de (1) (s'ils existent)?
\end{enumerate}

Bien entendu, ces questions ont peu de chances d'avoir
une r\'eponse int\'eressante dans une telle g\'en\'eralit\'e. Pr\'ecisons donc que les pseudo-foncteurs auxquels je
m'int\'eresse sont de type motivique: par exemple les motifs de Chow, les $1$-motifs, une
cat\'egorie de motifs triangul\'es de Voevodsky\dots\ mais aussi d'autres. Le premier exemple non trivial dans cette direction est la $K/k$-trace des variétés abéliennes.

Le premier probl\`eme est un probl\`eme d'existence, le second est un probl\`eme de
changement de base. Nous allons voir que le premier a une r\'eponse positive dans un grand nombre de cas int\'eressants, et que la r\'eponse au second n'est pas toujours celle qu'on
pourrait penser.

Concernant (1), on utilisera essentiellement trois m\'ethodes pour prouver l'existence d'un
adjoint:

\begin{itemize}
\item Un argument cat\'egorique abstrait.
\item Un calcul concret.
\item Supposant $K/k$ de type fini (condition plus ou moins n\'ecessaire pour l'existence d'un
adjoint), soit $U$ un mod\`ele de $K$ de type fini sur $k$. On suppose que $\sM$ s'\'etend en un
(pseudo-)foncteur sur la cat\'egorie des $k$-sch\'emas de type fini: au morphisme $f:U\to \Spec
k$ est donc associ\'e un foncteur
\[f^*:\sM(k)\to \sM(U).\]
Supposons que $f^*$ ait un adjoint (\`a gauche pour fixer les id\'ees) $f_\sharp$. On peut faire
varier $U$. Si $\sM(K)=2-\colim \sM(U)$, on a des chances d'obtenir l'adjoint de $\sM(k)\to
\sM(K)$ comme ``limite" des $f_\sharp$, cf. proposition \ref{p.2colim} pour un cas particulièrement simple.
\end{itemize}

Concernant (2), si $\phi$ est pleinement fidèle, le morphisme de changement de base \eqref{eq1.1} associé est en général un isomorphisme (lemme \ref{l3.0}). La situation est totalement différente quand $\phi$ est, disons, plein et essentiellement surjectif mais pas fidèle: passant de $\sM$ à $\sN$, on a alors tendance à ``perdre de l'adjoint'', cf. proposition \ref{p3.1} et remarque \ref{r8.1}.

Passons aux principaux résultats de l'article.

\begin{Th} Soit $\sM=\Mot_\num(-,\Q)$ la théorie des motifs numériques à coefficients rationnels.  Alors, pour toute extension primaire  $K/k$,  le foncteur 
\[\Mot_\num(k,\Q)\to \Mot_\num(K,\Q)\]
admet un adjoint \`a gauche et un adjoint \`a droite. Ces adjoints sont ca\-no\-ni\-que\-ment isomorphes, commutent aux twists et respectent les poids.
\end{Th}

Voir th. \ref{t1}, prop. \ref{p2} et prop. \ref{p5.1}. En se restreignant aux motifs effectifs de poids $1$, on retrouve la $K/k$-trace et la $K/k$-image des varétés abéliennes (à isogénie près): ex. \ref{ex5.1}. 

\begin{Th} Soit $p=1$ ou un nombre premier, et soit $\sM=S_b^{-1}\Sm$, $\Chow^\o(-,\Z[1/p])$ ou $\DM_\gm^\o$ la théorie des variétés lisses birationnelles \cite{birat}, des motifs de Chow birationnels à coefficients $\Z[1/p]$ \cite{birat-pure} ou des motifs triangulés géométriques birationnels \cite{birat-tri}. Alors, pour toute extension de type fini séparable  $K/k$ de corps d'exposant caractéristique $p$,  le foncteur 
\[\sM(k)\to \sM(K)\]
admet un adjoint à gauche.
\end{Th}

Voir th. \ref{t3}, \ref{t4} et \ref{t5}.

Une application de ces résultats a déjà été donnée dans \cite[prop. 5.1.4]{birat-pure}. Pour le corollaire suivant, on dit qu'une $k$-variété propre $X$ est \emph{universellement $R$-triviale} (resp. \emph{universellement $CH_0\otimes A$-triviale}) si $X(E)/R$ est réduit à un point pour toute extension $E/k$, où $R$ est la $R$-équivalence (resp. si $CH_0(X_E)\otimes A\iso A$ pour toute extension $E/k$, où $A$ est un anneau commutatif). Soit $p=1$ ou un nombre premier.

\begin{Coro} Soit $\pi:\sX\tto Y$ une application rationnelle dominante de $k$-variétés propres, lisses et connexes, dont la restriction à un ouvert de définition est propre. Supposons que la fibre générique $X$ de $\pi$ soit universellement $R$-triviale (resp. universellement $CH_0\otimes A$-triviale, où $A$ est une $\Z[1/p]$-algèbre). Alors $\pi$ induit un isomorphisme $\sX(L)/R\iso Y(L)/R$ (resp. $CH_0(\sX)\otimes A\iso CH_0(Y)\otimes A$) pour toute extension $L/k$  de corps d'exposant caractéristique $p$. 
\end{Coro}

Voir cor. \ref{c6.1a} et \ref{c6.2}. En particulier, si $X$ et $Y$ sont universellement $R$-triviales, il en est de même de $\sX$: c'est une variante d'un célèbre théorème de Graber, Harris et Starr \cite[cor. 1.3]{ghs}, mais ici le corps de base est quelconque. Tandis que le présent contexte catégorique n'est pas trop difficile à éviter pour traiter le cas de $CH_0$, en utilisant le lemme de déplacement pour les $0$-cycles \cite{vial}, cela semble moins évident dans le cas des classes de $R$-équivalence.

Le théorème qui suit construit un ``motif de Tate-\v Safarevi\v c'', introduit dans \cite{surfsurcourbe} de manière différente (voir remarque \ref{r9.2}). Notons $\Chow^\eff$ (resp. $\Mot_H^\eff$) la théorie des motifs de Chow effectifs (resp. celle des motifs effectifs modulo l'équivalence homologique relative à une cohomologie de Weil $H$). 

\begin{Th}\label{T3} Soit $K/k$ un corps de fonctions d'une variable, et soit $A$ une variété abélienne sur $K$. Alors il existe un motif $\sha(A,K/k)\in \Chow^\eff(k,\Q)$, dont l'image dans $\Mot_H^\eff(k,\Q)$ est de poids $2$ pour toute cohomologie de Weil $H$ classique (\S \ref{s2.2.2}), et tel que 
\[H_l(\sha(A,K/k))\simeq V_l(\Sha(A^*,K k_s))(-1)\]
pour tout nombre premier $l\ne \car k$, où $\Sha(A^*,K k_s)$ est le groupe de Tate-\v Safarevi\v c ``géométrique'' de la duale de $A$ et $H_l$ est la réalisation $l$-adique. Ce motif est fonctoriel en $A$ pour les homomorphismes de variétés abéliennes.
\end{Th}

Voir cor. \ref{c8.2} et th. \ref{t8.1}.

Le dernier résultat a trait aux ``correspondances au point générique'' étudiées par Bloch, Rovinsky et Beilinson. Pour toute $k$-variété projective lisse $X$, on note $\End_\gen(X)$ la $\Q$-algèbre de ces correspondances \cite{beilinson}.

\begin{Th} Soient $K/k$ une extension de type fini, $X$ une $K$-variété projective lisse et $\sX$ un $k$-modèle projectif lisse de $X$. Alors il existe un homomorphisme canonique d'algèbres
\[\End_\gen(X)\to \End_\gen(\sX)\]
dont on peut expliciter le conoyau.
\end{Th}

Voir th. \ref{t9.1}.

En particulier, si $\degtr(K/k)=1$ et si $X$ est une courbe (donc $\sX$) est une surface), $\End_\gen(X)$ est une $\Q$-algèbre semi-simple de dimension finie; ceci donne une approche possible de la conjecture de Bloch sur les surfaces (rem. \ref{r9.1}).
\bigskip

Dans la première section, on fait des ``rappels'' sur les théories motiviques et on fixe quelques notations. La seconde section est une liste d'exemples qui interviennent dans la suite de l'article. La troisième est une collection de sorites sur les adjonctions. Dans la section \ref{s4}, on traite le cas d'une extension finie: ces résultats ne sont ni difficiles ni surprenants (pour certaines théories motiviques, ils sont des cas extrêmement particuliers de ceux de \cite{ayoub}), mais ils ont déjà une application non triviale à un résultat d'effectivité pour les motifs purs: le théorème \ref{t2}.

La section \ref{s5} est consacrée aux motifs numériques. Vu le théorème de semi-simplicité de Jannsen \cite{jannsen}, l'existence des adjoints est ici un cas particulier de leur existence pour tout foncteur pleinement fidèle entre catégories semi-simples (proposition \ref{p1}). 

Dans la section \ref{s6}, on s'attaque aux catégories birationnelles étudiées dans \cite{birat,birat-pure,birat-tri}. Ici, la technique utilisée pour montrer l'existence d'un adjoint à gauche est celle de la proposition \ref{p.2colim}. Elle marche de manière remarquablement aisée et uniforme. On donne aussi d'autres démonstrations, fondées sur le calcul explicite des ensembles de morphismes dans ces catégories et qui fournissent peut-être une meilleure intuition de ce qui se passe; ces démonstrations ont toutefois l'inconvénient d'utiliser la résolution des singularités, sous la forme de Hironaka ou sous celle de de Jong.  

Dans la section \ref{s7}, on donne des exemples où les adjoints n'existent pas: extensions algébriques infinies, extension régulières infinies pour les motifs de Chow effectifs. Ce dernier cas (théorème \ref{p10.1}) n'est pas entièrement trivial et repose sur le fait qu'une variété abélienne non nulle a un rang positif sur un corps algébriquement clos de dimension de Kronecker $>0$. Le cas où la base de l'extension est  algébrique sur un corps fini est d'ailleurs une exception, voir proposition \ref{p7.1}.

Dans la section \ref{s8}, on fait quelques calculs explicites d'adjoints dans le cas des motifs de Chow birationnels. On en déduit le corollaire \ref{c8.2} (et donc une partie du théorème \ref{T3}).

Enfin, dans la section \ref{s9}, on applique ces résultats aux catégories de ``motifs au point générique'' développées dans \cite[\S 7.8]{kmp} à partir de l'idée de Bloch-Rovinsky-Beilinson.

Il reste à décrire ce qui n'est pas fait dans cet article. J'avais initialement prévu de traiter les $1$-motifs: cela s'avère beaucoup plus délicat que je ne le pensais, et devra faire l'objet d'un autre travail. Dans le cas des motifs purs, l'existence ou l'inexistence d'adjoints n'est claire que pour l'équivalence rationnelle et l'équivalence numérique: je ne sais rien dans le cas d'une relation intermédiaire (voir \S \ref{5.7} pour le cas des motifs homologiques). Dans le cas des motifs effectifs, un analogue du théorème \ref{p10.1} pour l'équivalence algébrique semble particulièrement intéressant à décider: je soup\c conne que le groupe de Griffiths donne une obstruction à l'existence d'un adjoint, mais je n'ai pas réussi à fabriquer un argument convaincant pour le démontrer. Enfin, la démonstration du théorème \ref{t8.1} est incomplète.

\enlargethispage*{30pt}

\subsection*{Remerciements} Ces r\'esultats ont \'et\'e initialement inspir\'es par mon
travail avec R. Sujatha sur les motifs birationnels \cite{Birat,localisation,birat,birat-pure,birat-tri}. Une partie de cet article a
\'et\'e imagin\'ee au cours d'un s\'ejour \`a l'IMPA en ao\^ut 2008, financ\'e par la
Coop\'eration France-Br\'esil. Je tiens \`a remercier cette derni\`ere de son soutien, l'IMPA
de son hospitalit\'e, et Marc Hindry et Am\'\i lcar Pacheco pour des \'echanges lors de ce
s\'ejour. Je remercie \'e\-ga\-le\-ment Ahmed Abbes, Joseph Ayoub, Daniel Bertrand, Mikhail Bondarko, Antoine Chambert-Loir,
Denis-Charles Cisinski, Pierre Colmez, Fr\'ed\'eric D\'eglise, Mathieu Florence, V. Srinivas, Claire Voisin 
et J\"org Wildeshaus pour des discussions autour de cet article. Finalement, je remercie le rapporteur pour ses commentaires qui m'ont permis d'améliorer la présentation.

\section{Notations et conventions}

\subsection{} Pour toute catégorie $\sA$, on note $Ob(\sA)$ (resp. $Fl(\sA)$) la classe des objets (resp. des morphismes) de $\sA$; si $X,Y\in \sA$, on note $\Hom_\sA(X,Y)$ ou $\sA(X,Y)$ l'ensemble des morphismes de $X$ vers $Y$, en choisissant la notation la plus pratique selon le contexte.

\subsection{} Soit $\corps$ la cat\'egorie des corps commutatifs (morphismes: les extensions de
corps). On se donne une sous-cat\'egorie $\sK\subset\corps$. Exemples: la sous-cat\'egorie
pleine des corps d'une caract\'eristique donn\'ee, la sous-cat\'egorie non pleine form\'ee des
extensions finies s\'eparables, etc.

\subsection{} Soit $\sC$ une $2$-cat\'egorie (par exemple la $2$-cat\'egorie $\Cat$ des 
cat\'egories, la $2$-cat\'egorie des cat\'egories additives, celle des
cat\'egories triangul\'ees, celle des cat\'egories $\Q$-lin\'eaires tensorielles rigides\dots).
On s'int\'eresse aux pseudo-foncteurs covariants \cite[VI.8]{SGA1}
\[\sM:\sK\to \sC.\]

Un tel pseudo-foncteur consiste donc en

\begin{itemize}
\item un objet $\sM(k)\in \sC$ pour tout $k\in \sK$;
\item un $1$-morphisme $\sM(f)=f_\sM^*=f^*:\sM(k)\to \sM(K)$ pour toute extension $f:k\to K$ de
$\sK$;
\item \'etant donn\'e deux extensions composables $f:k\to K$ et $g:K\to L$, un
$2$-\emph{isomorphisme}
\[c^\sM_{f,g}=c_{f,g}:(g\circ f)^*\iso g^*\circ f^* \]
o\`u les $c_{f,g}$ satisfont aux relations de $2$-cocycle habituelles. (En particulier, $(1_k)^*=Id_{\sM(k)}$.)
\end{itemize}

On dira que $\sM$ est une \emph{th\'eorie motivique} (d\'efinie sur $\sK$, \`a valeurs dans
$\sC$). On précisera ``théorie motivique additive, triangulée'', etc. si $\sC$ est (etc.).

\subsection{}\label{s1.3} Si $\sM,\sN:\sK\to \sC$ sont deux th\'eories motiviques, un
\emph{morphisme de th\'eories motiviques}
$\phi:\sM\to\sN$ est donn\'e par
\begin{itemize}
\item Pour tout $k\in\sK$, un $1$-morphisme $\phi_k:\sM(k)\to \sN(k)$.
\item Si $f:k\to K$ est un morphisme de $\sK$, un $2$-\emph{isomorphisme}
\[v_f^\phi=v_f:\phi_K f_\sM^*\iso f_\sN^*\phi_k\]
satisfaisant aux relations de $1$-cocycle habituelles (relativement \`a $c^\sM$ et $c^\sN$).
\end{itemize}

\enlargethispage*{30pt}

\subsection{} Nous n'utiliserons pas la notion d'(iso)morphisme entre mor\-phis\-mes de
th\'eories motiviques (qui existe).

\subsection{} Rappelons que la notion d'adjoint d'un foncteur se
g\'en\'eralise directement au cas des $1$-morphismes dans une $2$-cat\'egorie
\cite[1.1]{ayoub}. Un adjoint est d\'etermin\'e \`a $2$-isomorphisme unique pr\`es. \'Etant
donn\'e une th\'eorie motivique
$\sM$ d\'efinie sur $\sK$, nous rencontrerons de nombreux exemples
o\`u les morphismes
$f^*_\sM$ admettent un adjoint \`a gauche $f_\sharp^\sM$ (ou \`a droite $f_*^\sM$). Rappelons que, dans ce cas, on peut
relier ces adjoints de mani\`ere coh\'erente, de fa\c con \`a munir $\sM$ d'une structure
$1$-contravariante sur $\sK$ \cite[prop. 1.1.17]{ayoub}.

Si $f_\sharp^\sM$ existe pour tout $f$, on dira que $\sM$ \emph{admet des adjoints \`a gauche}.

\subsection{}\label{1.6} Soient $\sM,\sN:\sK\to \sC$ deux th\'eories motiviques, et soit $\phi:\sM\to \sN$
un morphisme de th\'eories motiviques. Supposons que, pour $f:k\to K$ dans $\sK$, $f_\sM^*$
et $f_\sN^*$ admettent des adjoints \`a gauche $f^\sM_\sharp$ et $f^\sN_\sharp$. On a donc des
$2$-morphismes d'adjonction
\begin{align*}
1\by{\eta_\sM}f_\sM^* f^\sM_\sharp,\qquad &f^\sM_\sharp f_\sM^*\by{\epsilon_\sM} 1\\ 
1\by{\eta_\sN}f_\sN^* f^\sN_\sharp,\qquad &f^\sN_\sharp f_\sN^*\by{\epsilon_\sN} 1 
\end{align*}
satisfaisant les relations habituelles. On d\'eduit alors du $2$-isomorphisme $v^\phi_f$ de \ref{s1.3} le
\emph{morphisme de changement de base} \cite[prop. 1.1.9]{ayoub}
\begin{equation}\label{eq1.1}
w_f^\phi: f^\sN_\sharp\phi_K\to \phi_k f^\sM_\sharp
\end{equation}
défini comme la composition
\[
f^\sN_\sharp\phi_K\by{f^\sN_\sharp\phi_K*\eta_\sM}f^\sN_\sharp\phi_K f^*_\sM f_\sharp^\sM\by{f^\sN_\sharp*v_f^\phi*f_\sharp^\sM} f^\sN_\sharp f^*_\sN\phi_k  f_\sharp^\sM\by{\epsilon_\sN * \phi_k f^\sM_\sharp} \phi_k f^\sM_\sharp.
\]

\section{Exemples} 

Tous ces exemples sont utilisés dans la suite de l'article.

\subsection{Le mod\`ele de base} Pour tout $k\in \corps$, soit $\Sm(k)\in\Cat$ la cat\'egorie des
$k$-schémas lisses séparés de type fini,   la structure de $2$-foncteur étant donnée par l'extension des scalaires. Les cat\'egories $\Sm(k)$ poss\`edent produits et coproduits finis. {\bf Variantes}: prendre les vari\'et\'es projectives lisses $\Sm^\proj$,  etc. On peut aussi prendre les vari\'et\'es lisses (ou projectives lisses) point\'ees (par un point  rationnel): cela d\'efinit de nouvelles th\'eories motiviques $\Sm_\bullet,\Sm_\bullet^\proj$\dots

\subsection{Motifs purs}\label{s2.2} Choisissons un \emph{couple ad\'equat} $(A,\sim)$ au sens de \cite[1.1]{birat-pure}: $A$ est
un anneau commutatif et $\sim$ est une relation d'\'equivalence ad\'equate sur les cycles
alg\'ebriques
\`a coefficients dans $A$. On a les th\'eories motiviques suivantes, munies de morphismes de
th\'eories motiviques:
\[\Sm^\proj\to \Cor_\sim(-,A)\to \Mot_\sim^\eff(-,A)\to \Mot_\sim(-,A)\]
o\`u $\Cor_\sim$ (resp. $\Mot_\sim^\eff,\Mot_\sim$) d\'esigne la cat\'egorie des correspondances
(resp. des motifs purs effectifs, des motifs purs) avec la convention \emph{covariante}. Lorsque $\sim$ est l'équivalence rationnelle, on abrège $\Mot_\rat^\eff$ et $\Mot_\rat$ en $\Chow^\eff$ et $\Chow$.

Pour que cette d\'efinition ait un sens, il faut que la relation $\sim$ soit compatible \`a
l'extension des scalaires. C'est clairement le cas pour la plupart des \'equivalences
ad\'equates usuelles: par exemple \'equivalence rationnelle, alg\'ebrique, num\'erique ou
smash-nilpotente au sens de Voevodsky. Dans le cas de l'\'equivalence homologique, on est
amen\'e \`a faire un peu de th\'eorie:

\subsubsection{Une th\'eorie motivique de r\'ealisations} Nous aurons besoin d'un exemple tr\`es
sommaire de telle th\'eorie:

\begin{defn}\phantomsection
Pour $p=0$ ou un nombre premier, soit $\corps_p$ la cat\'egorie des corps de
caract\'eristique $p$. Une \emph{th\'eorie de coefficients} est un foncteur
\[F:\corps_p\to \corps_0.\]
\end{defn}

\begin{exs}\phantomsection\
\begin{thlist}
\item Un foncteur constant: $F(k) = F_0$ pour tout $k$, les morphismes \'etant \'egaux \`a
l'identit\'e.
\item $F(k) = k$ si $p=0$. 
\item $F(k)= Frac (W(k))$ si $p>0$.
\end{thlist}
\end{exs}

\`A une th\'eorie de coefficients $F$, on associe la th\'eorie motivique
\[R_F(k) = \Vec_{F(k)}^*\]
o\`u $\Vec_{F(k)}^*$ est la cat\'egorie des $F(k)$-espaces vectoriels gradu\'es de dimension
finie. Pour une extension $K/k$, le foncteur $\Vec_{F(k)}^*\to \Vec_{F(K)}^*$  est
donn\'e par l'extension des scalaires. (On consid\`ere les $\Vec^*$ comme des cat\'egories
rigides pour le produit tensoriel gradu\'e et la contrainte de commutativit\'e donn\'ee par la
r\`egle de Koszul.)

\subsubsection{Th\'eories motiviques de motifs homologiques}\label{s2.2.2} Les cohomologies de
Weil ``classiques" fournissent des exemples de morphismes $\Chow\to R_F$:

\begin{thlist}
\item Fixons un domaine universel $\Omega\in \corps_p$ et ``restreignons''-nous \`a la cat\'egorie des sous-corps de $\Omega$ 
(c'est-\`a-dire considérons la cat\'egorie $\corps/\Omega$). Pour $l\ne p$, on consid\`ere la r\'ealisation $l$-adique $X\mapsto H^*(X_{k_s},\Q_l)$, où $k_s$ est la clôture séparable de $k$ dans $\Omega$.
Ici, $F(k)=\Q_l$ pour tout $k$. L'invariance de la cohomologie \'etale par extension
alg\'ebriquement close fournit les isomorphismes naturels n\'ecessaires. On note cette réalisation $\rho_l$, et $\Mot_l$ la théorie de motifs homologiques associée.
\item Prenons $p=0$, et restreignons-nous \`a la cat\'egorie des sous-corps de $\C$. On peut alors prendre pour $H^*$ la
cohomologie de Betti. Ici, $F(k)=\Q$ pour tout $k$. On note cette réalisation $\rho_B$, et $\Mot_B$ la théorie de motifs purs homologiques associée.
\item Toujours en caract\'eristique z\'ero, on peut prendre pour $H^*$ la cohomologie de de
Rham. Ici, $F(k) = k$. On note cette réalisation $\rho_{dR}$ et $\Mot_{dR}$ la théorie de motifs purs homologiques associée.
\item En caract\'eristique $p>0$, on peut prendre pour $H^*$ la cohomologie cristalline. Ici,
$F(k)=Frac(W(k))$. On note cette réalisation $\rho_{cris}$. et $\Mot_{cris}$ la théorie de motifs purs homologiques associée.
\end{thlist}

Rappelons que si $p=0$, les isomorphismes de comparaison montrent que les théories $\Mot_l$, $\Mot_B$ et $\Mot_{dR}$ sont \emph{égales}. On les note simplement $\Mot_H$.

\subsubsection{Réalisations enrichies}\label{s2.2.3} Je ne traiterai que le cas que je connais bien: celui de la cohomologie $l$-adique, laissant d'autres exemples 
aux experts. Si $G$ est un groupoïde et $F$ un corps, on note $\Rep_F(G)$ (resp. $\Rep_F^*(G)$) la catégorie des foncteurs de $G$ vers $\Vec_F$ (resp. $\Vec_F^*$): ce sont les \emph{$F$-représentations} (de dimension finie, graduées) de $G$. Si $G$ et $F$ sont munis de topologies, la notation sous-entend qu'on ne considère que les représentations continues (c'est-à-dire continues sur les ensembles de morphismes).

\begin{defn}\phantomsection\label{d2.1} Soit $p=0$ ou un nombre premier, et soit $l\ne p$  un nombre premier. Pour $k\in \corps_p$, de type fini sur son sous-corps premier, on définit:
\[F_l(k) = \Rep_{\Q_l}^*(G_k)  \]
où $G_k$ est le groupoïde de Galois absolu de $k$.\footnote{Rappelons que la catégorie $G_k$ a pour objets les clôtures séparables de $k$ et pour morphismes les $k$-isomorphismes. 
Le choix d'une clôture séparable de $k$ identifie ces données aux représentations $l$-adiques habituelles, mais on perd alors la fonctorialité en $k$. Je remercie Joseph Ayoub d'avoir attiré mon attention sur ce point.} Ceci définit une théorie motivique sur la sous-catégorie pleine de $\corps_p$ formée des corps de type fini. On l'étend à tout $\corps_p$ en posant
\[F_l(k) = 2-\colim_{k'\subset k} F_l(k')\]
où $k'$ décrit les sous-corps de $k$ de type fini sur le corps premier. On obtient alors un morphisme de théories motiviques 
\[\tilde \rho_l:\Mot_l \to F_l: \quad X/k\mapsto (k_s\mapsto H^*(X_{k_s},\Q_l))\]
(réalisation enrichie). Le choix d'un domaine universel $\Omega$ fournit un ``morphisme d'oubli'' $\omega_\Omega:F_l\to R_{\Q_l}$ sur $\corps_p/\Omega$ tel que $\rho_l=\omega_\Omega \circ \tilde \rho_l$.
\end{defn}

\begin{rque}\phantomsection\label{r2.1} On a une sous-théorie motivique pleine de $F_l$:
\[F_l^{ss}(k) = \{V\in F_l(k)\mid V \text{ est semi-simple}\}\]
(pour $k$ de type fini; on l'étend ensuite par $2$-limite inductive). La conjecture de semi-simplicité de Grothendieck-Serre prédit que \emph{$\tilde R_l$ prend ses valeurs dans $F_l^{ss}$}. 
\end{rque}

\subsection{Correspondances finies} Les catégories de correspondances finies de Voevodsky \cite{voetri} définissent une théorie motivique $\Corf$.

\subsection{Th\'eories de pr\'efaisceaux et de faisceaux} Si $\sM$ est une th\'eorie motivique, 
$\hat\sM(k)=\Hom(\sM(k)^{op},\Ens)$ d\'efinit une nouvelle th\'e\-o\-rie motivique $\hat\sM$, 
en prenant  
\[\hat \sM(f)=\sM(f)_!\]
adjoint à gauche du foncteur  $\sM(f)^*: \hat \sM(K)\to \hat \sM(k)$ défini par
\[\sM(f)^*(X)(M)= X(\sM(f)(M))\]
pour $f:k\to K$ (lemme \ref{l3.4}). Le foncteur de Yoneda d\'efinit un morphisme de th\'eories motiviques 
$y_\sM:\sM\to \hat\sM$.

On peut aussi choisir $k\mapsto \Hom(\sM(k)^{op},\sC)$, o\`u $\sC$ est une cat\'egorie cocompl\`ete, 
mais on n'a plus de foncteur de Yoneda. Exemples de $\sC$: groupes ab\'eliens, ensembles 
simpliciaux\dots\ Pour ces exemples on a un foncteur canonique $\Ens\to \sC$, d'o\`u un morphisme 
$\hat\sM\to \Hom(\sM^{op},\sC)$.

\subsubsection{Le cas additif}\label{2.4.1} Si $\sM$ est une théorie motivique additive, on remplace $\Ens$ ci-dessus par $\Ab$ (catégorie des groupes abéliens), en se restreignant aux préfaisceaux additifs; on a alors le plongement de Yoneda additif \cite[\S 1.3]{nrsm}.

\subsubsection{Faisceaux} Si $\sM(k)$ est munie pour tout $k$ d'une topologie de Grothendieck et 
que $f^*:\sM(k)\to \sM(K)$ est continu pour tout $f$, on peut consid\'erer la sous-th\'eorie pleine 
$\tilde\sM\subset \hat \sM$ des faisceaux. De m\^eme \`a coefficients dans une cat\'egorie convenable.

\subsection{Th\'eories homotopiques}\label{s2.3} D'autres th\'eories motiviques munies de
morphismes de th\'eories motiviques sont
\[\Sm\to\sH\to \sS\sH^\eff\to \sS\sH\]
o\`u, pour tout corps $k$, $\sH(k)$ (resp. $\sS\sH(k)$) est la cat\'egorie
homotopique (resp. homotopique stable) des sch\'emas de Morel-Voevodsky et $\sS\sH^\eff(k)$ est
la cat\'egorie not\'ee $\sS\sH^{\A^1}(k)$ dans \cite[\S 3.2]{morel}. Rappelons \cite{mv} que $\sH(k)$ est 
une cat\'egorie homotopique associ\'ee \`a la cat\'egorie des faisceaux Nisnevich d'ensembles 
simpliciaux sur $\Sm(k)$, et que $\sS\sH^\eff(k)$ et $\sS\sH(k)$ s'en d\'eduisent par stabilisation \cite{vICM}.

\subsection{Th\'eories triangul\'ees}\label{s2.4} On dispose \'egalement des cat\'egories
triangul\'ees de motifs d\'efinies par Voevodsky: par exemple
\[\begin{CD}
\Sm @>>> \DM_\gm^\eff@>>>  \DM^\eff\\
&& @VVV @VVV\\
&& \DM_\gm@>>> \DM.
\end{CD}\]

\subsection{} Il existe divers morphismes entre les exemples \ref{s2.2}, \ref{s2.3} et
\ref{s2.4}, dont nous rappellons les principaux:
\begin{align*}
\Chow^\eff&\to \DM_\gm^\eff\leftarrow  \Corf\\
\sS\sH^\eff&\to \DM^\eff.
\end{align*}

\subsection{Sch\'emas en groupes commutatifs} Pour $k\in \corps$, soit $\Grp(k)$ la cat\'egorie
des $k$-sch\'emas en groupes commutatifs localement de type fini. C'est une th\'eorie motivique  additive. Elle contient des 
sous-th\'eories pleines importantes:
\begin{itemize}
\item la th\'eorie $\SAb$ des vari\'et\'es semi-ab\'eliennes (extensions d'une vari\'et\'e
ab\'elienne par un tore);
\item la th\'eorie $\Ab$ des vari\'et\'es ab\'eliennes;
\item la th\'eorie $\tore$ des tores;
\item la th\'eorie $\res$ des r\'eseaux (sch\'emas en groupes localement isomorphes \`a
$\Z^n$ pour la topologie \'etale).
\end{itemize}

\subsection{Localisation} Soit $\sM$ une th\'eorie motivique \`a valeurs dans
$\Cat$. Supposons donn\'e, pour tout $k\in \sK$, un ensemble de morphismes $S(k)\in
Fl(\sM(k))$, contenant les isomorphismes et tels que $f^* S(k)\subset S(K)$ pour tout
$f:k\to K$ de $\sK$. Les cat\'egories localis\'ees $S(k)^{-1}\sM(k)$ forment alors une nouvelle
th\'eorie motivique, not\'ee $S^{-1}\sM$. On a un morphisme de th\'eories motiviques
\[\sM\to S^{-1}\sM.\]

Cette construction poss\`ede les variantes suivantes:

Si $\sM$ prend ses valeurs dans une $2$-cat\'egorie avec
structures sup\-pl\'e\-men\-tai\-res, on peut souhaiter obtenir des localisations conservant ces
structures suppl\'ementaires. Par exemple, si $\sM$ prend ses valeurs dans la $2$-cat\'egorie
des cat\'egories triangul\'ees, les cat\'egories $S(k)^{-1} \sM(k)$ ne sont pas en g\'en\'eral
triangul\'ees. Pour y rem\'edier, on peut utiliser la \emph{saturation triangul\'ee} de $S(k)$:
c'est le plus petit ensemble de morphismes $\langle S(k)\rangle$ de $\sM(k)$ contenant $S(k)$
et tel que la cat\'egorie
$\langle S(k)\rangle^{-1} \sM(k)$ soit triangul\'ee (cf. \cite[def. 4.2.1]{birat-tri}). Concr\`etement, on consid\`ere la
sous-cat\'egorie (strictement) \'epaisse $\sS$ de $\sM(k)$ engendr\'ee par les c\^ones des
\'el\'ements de $S(k)$, et on prend pour $\langle S(k)\rangle$ l'ensemble des morphismes dont le
c\^one appartient \`a $\sS$.

L'exemple principal ici est:

\subsubsection{Th\'eories birationnelles}\label{s.bir} Soit $\sM$ une th\'eorie
motivique \`a valeurs dans
$\Cat$, munie d'un morphisme de th\'eories motiviques $\phi:\Sm\to \sM$ ou $\phi:\Sm^\proj\to
\sM$. Notons $S_b$ l'ensemble des morphismes birationnels entre vari\'et\'es (projectives)
lisses. Le morphisme $\phi$ envoie $S_b$ sur des ensembles de morphismes que nous continuerons
\`a noter $S_b$. On obtient ainsi un carr\'e naturellement commutatif de th\'eories motiviques
\[\begin{CD}
\Sm@>>> \sM\\
@VVV @VVV\\
S_b^{-1}\Sm@>>> S_b^{-1}\sM
\end{CD}
 \text{ ou }
\begin{CD}
\Sm^\proj@>>> \sM\\
@VVV @VVV\\
S_b^{-1}\Sm^\proj@>>> S_b^{-1}\sM
\end{CD}
\]

Si $\sM$ est une th\'eorie motivique de cat\'egories triangul\'ees, on remplacera $S_b$ par 
$\langle S_b\rangle$.

\section{G\'en\'eralit\'es sur les adjoints}

\begin{lemme}\phantomsection\label{l3.0} Soit $T:\sM\to \sN$ un foncteur; soient $\sM'$ une sous-cat\'egorie 
pleine de $\sM$ et  $\sN'$ une sous-catégorie pleine de $\sN$ telles que $T(\sM')\subset \sN'$. Supposons que $T$ admette un adjoint 
(\`a gauche ou \`a droite) $S$ et que $S(\sN')\subset \sM'$. Alors le foncteur $T':\sM'\to \sN'$ induit
par $T$ a pour adjoint (\`a gauche ou \`a droite) le foncteur $S':\sN'\to \sM'$ induit par $S$. De plus, le morphisme de changement de base correspondant est un isomorphisme.\qed
\end{lemme}

\begin{lemme}\phantomsection\label{l3.2} Soit $T:\sM\to \sN$ un foncteur additif entre cat\'egories additives. 
Alors\\
a) La sous-cat\'egorie pleine $\sN_T$ de $\sN$ form\'ee des objets $N$ tels que l'adjoint 
\`a gauche (resp. \`a droite) $S$ de $T$ soit d\'efini en $N$ est additive; de plus, $S$ est 
un foncteur additif de $\sN_T$ vers $\sM$.\\
b) Si $\sM$ est pseudo-ab\'elienne, $\sN_T$ est \'epaisse 
(stable par facteurs directs repr\'esentables dans $\sN$).\\
c) 
Si $\sM$ est stable par sommes directes infinies 
et si $S$ est un adjoint \`a gauche, $\sN_T$ est 
stable par sommes directes infinies repr\'esentables dans 
$\sN$
. De plus,  
$S$ commute \`a ces sommes directes.
\end{lemme}

\begin{proof} a) Soient $M,N\in \sN$ tels que $S$ soit d\'efini en $M$ et $N$. Clairement, $S(M)\oplus S(N)$ 
repr\'esente $S(M\oplus N)$. Ceci montre que $\sN_T$ est additive et que $S$ est additif. 

b) Soit $N\in \sN_T$ 
et soit $P$ un facteur direct de $N$. Soit $e\in \End(N)$ l'idempotent d'image $P$. Alors $S(e)$ est un 
endomorphisme idempotent de $S(N)$: si $\sM$ est pseudo-ab\'elienne, il admet une image qui repr\'esente $S(P)$.

c) Soit $I$ un ensemble, et soit $(N_i)_{i\in I}$ une famille d'objets de $\sN_T$. 
Supposons que $\bigoplus N_i$ soit repr\'esentable dans $\sN$. Pour $M\in \sM$, on a
\[\sN(\bigoplus N_i,TM)=\prod \sN(N_i,TM)=\prod \sM(SN_i,M) = \sM(\bigoplus SN_i,M)\]
ce qui montre que $\bigoplus SN_i$ repr\'esente $S(\bigoplus N_i)$.
\end{proof}

\begin{lemme}[g\'en\'eralisation de \protect{\cite[lemma 5.3.6]{neeman}}]\phantomsection\label{l3.3}
Soit $T:\sM\to \sN$ un foncteur triangul\'e entre cat\'egories triangul\'ees. Alors la sous-cat\'egorie 
pleine $\sN_T$ de $\sN$ form\'ee des objets $N$ tels que l'adjoint \`a gauche (resp. \`a droite) $S$ de $T$ soit 
d\'efini en $N$ est triangul\'ee. De plus, $S$ est triangul\'e sur son domaine de d\'efinition.
\end{lemme}

\begin{proof} Que $\sN_T$ soit triangul\'ee se d\'emontre comme dans la preuve de 
\cite[lemma 5.3.6]{neeman}: si $N\in \sN$, on v\'erifie que $N[i]\in \sN$ pour tout $i\in\Z$, et si 
$f:M\to N$ est un morphisme de $\sN_T$ de cône $P$, on v\'erifie qu'un cône de $S(f)$ repr\'esente $S(P)$. 
La preuve montre que $S$ est triangul\'e. 
\end{proof}

\begin{lemme}[\protect{\cite[I.5.3]{SGA4}}]\phantomsection\label{l3.4} Soit $T:\sM\to \sN$ un foncteur. Alors le foncteur $T^*:\hat \sN\to \hat \sM$ ``composition avec $T$'' entre les 
cat\'egories de pr\'efaisceaux d'ensembles a un adjoint \`a gauche $T_!$. Si $T$ a un adjoint \`a gauche $S$, $T_!$ 
a pour adjoint \`a gauche $S_!$. Si $T$ a un adjoint \`a droite $U$, alors $U_!\simeq T^*$. M\^emes 
\'enonc\'es pour les pr\'efaisceaux \`a valeurs dans une cat\'egorie cocompl\`ete.\qed
\end{lemme}

Rappelons qu'une \emph{dualité} sur une catégorie $\sM$ est un foncteur $*:\sM\to \sM^{op}$ tel que $*^{op}$ soit adjoint (à droite, et donc à gauche) de $*$. Cette dualité est \emph{parfaite} si $*$ est une équivalence de catégories; cf. \cite[\S 1]{qss} dans le cas additif.

\begin{lemme}\phantomsection\label{l3.5} Soient $\sM,\sN$ deux cat\'egories munies de dualit\'es parfaites $^*$ , et soit $T:\sM\to \sN$ un foncteur 
commutant \`a ces dualit\'es. Soit $N\in\sN$. Alors l'adjoint \`a gauche $S$ de $T$ est d\'efini en $N$ si et 
seulement si l`adjoint \`a droite $U$ de $T$ est d\'efini en $N$. Plus pr\'ecis\'ement, on a la formule
\[U(N) = S(N^*)^*.\qed\]
\end{lemme}

\begin{prop}\phantomsection\label{p2bis} Soient $\sM,\sN$ deux cat\'egories mono\"\i dales sy\-m\'e\-tri\-ques rigides, et soit 
$T:\sM\to \sN$ un foncteur mono\"\i dal sym\'etrique. Supposons que $T$ admette un adjoint \`a gauche 
$S$. Alors, pour tout $(M,N)\in \sM\times \sN$, on a la
\emph{formule de projection}
\[S(N\otimes TM)\simeq SN\otimes M.\]
\end{prop}

\begin{proof} En utilisant le fait que $T$ pr\'eserve le produit tensoriel et la
dualit\'e, on a des isomorphismes pour $P\in \sM$:
\begin{align*}
\sM(S(N\otimes TM),P)&\simeq \sN(N\otimes TM,TP)\\
&\simeq\sN(N,(TM)^*\otimes TP)\\
&\simeq \sN(N,T(M^*\otimes P))\\ 
&\simeq\sM(SN,M^*\otimes P)\\
&\simeq\sM(SN\otimes M,P).
\end{align*}

La conclusion en r\'esulte par Yoneda.
\end{proof}

\begin{prop}\phantomsection\label{p.2colim} a) Soit $(\sC_\alpha)_{\alpha\in A}$ un $2$-système inductif de catégories (c'est-à-dire un pseudo-foncteur $A\to \Cat$), de $2$-colimite $\sC$. Pour tout $\alpha$, soit $S_\alpha$ une classe de morphismes de $\sC_\alpha$, et supposons que tout morphisme $\alpha\to \beta$ de $A$ envoie $S_\alpha$ dans $S_\beta$. Soit $S=\bigcup_\alpha \IM(S_\alpha)\subset Fl(\sC)$. Alors on a une équivalence canonique de catégories
\[S^{-1} \sC \simeq 2-\colim S_\alpha^{-1} \sC_\alpha.\]
b) Avec les notations précédentes, soient $\sC_0$ une autre catégorie et, pour tout $\alpha$, $f_\alpha:\sC_0\to \sC_\alpha$ un foncteur. Soit $S_0$ une classe de morphismes de $\sC$. On suppose que les $f_\alpha$ envoient $S_0$ dans $S_\alpha$ et sont compatibles aux foncteurs de transition, d'où un foncteur induit $f:S_0^{-1}\sC_0\to S^{-1}\sC$. On suppose de plus que:
\begin{thlist}
\item Pour tout $\alpha$, $f_\alpha$ admet un adjoint à gauche qui envoie $S_\alpha$ dans $S_0$.
\item Pour tout morphisme $\alpha\to \beta$, le foncteur $S_\alpha^{-1}\sC_\alpha \to S_\beta^{-1}\sC_\beta$ est une équivalence de catégories.
\end{thlist}
Alors $f$ admet un adjoint à gauche.
\end{prop}

\begin{proof} a) Les deux membres de l'équivalence ont chacun la propriété 2-universelle de l'autre. b) résulte immédiatement de a).
\end{proof}

Dans le dernier énoncé, nous nous cantonnons au cadre des théories motiviques pour éviter de perdre le lecteur dans une hypergénéralité inutile.

\begin{prop}\phantomsection\label{p3.1} Dans la situation du \S \ref{1.6},  supposons $\phi_K$ plein et $\phi_k$ essentiellement surjectif. Alors, pour tout $M\in \sM(K)$, le morphisme de changement de base $w_f^\phi(M): f_\sharp^\sN\phi_K (M)\to \phi_k f_\sharp^\sM (M)$ de \eqref{eq1.1} est un monomorphisme scindé.
\end{prop}

\begin{proof} Soit $N\in \sM(k)$. \'Evaluant $w_f^\phi(M)$ sur $\phi_k(M)$, on obtient:
\[ \sN(k)(\phi_k f_\sharp^\sM (M),\phi_k(N))\by{w_f^\phi(M)^*}\sN(k)(f_\sharp^\sN\phi_K (M),\phi_k(N)).
\]

Par adjonction et commutation des $f^*$ aux $\phi$, le second membre se récrit:
\[\sN(K)(\phi_K (M),f^*_\sN\phi_k(N))\simeq \sN(K)(\phi_K (M),\phi_Kf^*_\sM(N)).\]

On en déduit un diagramme commutatif:
\[\begin{CD}
\sM(k)(f_\sharp^\sM (M),N)@>\sim>>\sM(K)(M,f^*_\sM(N)) \\
@V\phi_k VV @V``\phi_K" VV \\
\sN(k)(\phi_k f_\sharp^\sM (M),\phi_k(N))@>{w_f^\phi(M)^*}>>\sN(k)(f_\sharp^\sN\phi_K (M),\phi_k(N))
\end{CD}\]
où la flèche horizontale supérieure est l'isomorphisme d'adjonction et la flèche verticale de droite est surjective par hypothèse. Il en résulte que la flèche horizontale inférieure est \emph{surjective}. Par hypothèse sur $\phi_k$, on peut choisir $N$ tel que $\phi_k(N)\simeq f_\sharp^\sN\phi_K (M)$. En relevant cet isomorphisme, on obtient la rétraction cherchée. 
\end{proof}

\section{Cas d'une extension finie}\label{s4}

\subsection{Cas d'une extension finie s\'eparable}

Soit $f:k\to K$ une extension finie s\'eparable.

\begin{thm}\phantomsection\label{t4.1}  $f_\sharp$ existe pour les 
th\'eories motiviques suivantes: $\Sm^\proj$, $\Sm$, $\Cor_\sim(-,A)$, $\Corf$, $\Mot_\sim^\eff(-,A)$,
 $\Mot_\sim(-,A)$, $\sH$, $\sS\sH^\eff$, $\sS\sH$, $\DM_\gm^\eff$, $\DM_\gm$, $\DM_-^\eff$, $\DM^\eff,\DM$. 
De plus, les changements de base \eqref{eq1.1} 
relatifs aux divers morphismes entre ces th\'eories motiviques sont des isomorphismes.
\end{thm}

\begin{proof} Dans le cas de $\Sm$ et $\Sm^\proj$, $f_\sharp$ est donn\'e par la restriction ``na\"ive'' 
des scalaires. Dans le cas de $\Cor_\sim(-,A)$, la formule donnant les Hom montre que $f_\sharp h(X)$ 
existe et est donn\'e par $h(f_\sharp X)$ pour $X\in \Sm^\proj(K)$ (remarquer que, pour 
$Y\in \Sm^\proj(k)$, $X\times_K Y_K = f_\sharp X\times_k Y$). Même raisonnement pour $\Corf$. D'apr\`es le lemme \ref{l3.2}, 
cela d\'efinit $f_\sharp$ sur $\Mot_\sim^\eff(-,A)$ tout entier. De m\^eme, dans le cas de $\Mot_\sim(-,A)$, 
la meilleure mani\`ere de prouver l'existence de $f_\sharp$ est de consid\'erer $\Mot_\sim(-,A)$ comme 
l'enveloppe karoubienne de la cat\'egorie des correspondances gradu\'ees, pour laquelle le raisonnement 
est le m\^eme que pour $\Cor_\sim(-,A)$. Ces constructions montrent que les morphismes de changement de 
base sont des isomorphismes.

Dans le cas de $\DM_\gm^\eff$, $\DM_\gm$, $\DM^\eff$, $\DM$, $\sS\sH^\eff$ et $\sS\sH$, on a bien s\^ur une fonctorialit\'e bien plus g\'en\'erale \cite{ayoub,cis-deg}\footnote{La notation $f_\sharp$ est inspir\'ee de \cite{ayoub}.}; donnons n\'eanmoins une d\'emonstration \'el\'ementaire. La sous-cat\'egorie 
pleine form\'ee des $M(X)$ pour $X\in\Sm$ est dense; de plus, les cat\'egories consid\'er\'ees v\'erifient 
les hypoth\`eses du lemme \ref{l3.2} b) et c). En appliquant ce lemme et le lemme \ref{l3.3}, on voit qu'il
suffit de prouver que $M(f_\sharp X)$ repr\'esente $f_\sharp M(X)$ dans $\DM(k)$. De plus, en appliquant le lemme \ref{l3.0}, il 
suffit de travailler dans $\DM$ pour attraper toutes ses sous-th\'eories pleines.

Soient donc $X\in \Sm(K)$ et $M\in \DM(k)$. Par d\'efinition, $M$ est donn\'e par un $\Z(1)$-spectre 
$(M_n,f_n)_{n\in\Z}$ o\`u $M_n\in \DM^\eff(k)$ et $f_n:M_n(1)\to M_{n+1}$; de plus
\[\DM(K)(M(X),M_K)=\colim_{n\ge 0} \DM^\eff(K)(M(X)(n),(M_n)_K).\] 

Pour montrer que $M(f_\sharp X)$ repr\'esente $f_\sharp M(X)$ dans $\DM(k)$, on est donc ramen\'e \`a 
d\'emontrer cette assertion dans $\DM^\eff(k)$. Pour $M\in \DM^\eff(k)$,
\begin{multline*}
\DM^\eff(K)(M(X),M_K)=H^0_\Nis(X,M_K)\\
=H^0_\Nis(X,M)=DM^\eff(k)(M(f_\sharp X),M)
\end{multline*} 
puisque $M_K$ est simplement la restriction de $M$ \`a $\Sm(K)_\Nis$ vu comme sous-site de $\Sm(k)_\Nis$ 
et que la cohomologie de Nisnevich de $X$ est in\-d\'e\-pen\-dan\-te de la base.

Pour $\sS\sH^\eff$ et $\sS\sH$, on proc\`ede de m\^eme. Pour $\sH$, l'idéal serait de d\'evelopper une notion abstraite de cat\'egorie homotopique instable et de g\'en\'eraliser le lemme \ref{l3.3} \`a ce cadre. \`A défaut, on peut raisonner directement et montrer que le foncteur $f_\sharp:\Sm(K)\to \Sm(k)$ induit bien un foncteur $f_\sharp:\sH(K)\to \sH(k)$, adjoint à gauche de $f^*$.
\end{proof}

\begin{thm}\phantomsection\label{t4.2}  $f_*$ existe pour les 
th\'eories motiviques suivantes: $\Sm^\proj$, $\Sm$, $\Sm^\proj_\bullet$, $\Sm_\bullet$, $\Grp$, $\SAb$, 
$\Ab$, $\tore$, $\res$
. De plus, les changements de base \eqref{eq1.1} relatifs aux divers 
morphismes entre ces th\'eories motiviques sont des isomorphismes.
\end{thm}

\begin{proof} Pour $\Sm$ et $\Sm^\proj$, $f_*$ est donn\'e (par définition!) par la restriction des scalaires \`a la Weil. 
Ceci fonctionne encore pour $\Sm_\bullet$, $\Sm^\proj_\bullet$ et $\Grp$, puisque ce foncteur commute 
aux produits (comme adjoint \`a droite). 
De plus, il est clair que $f_*$ pr\'eserve les sous-cat\'egories pleines $\SAb$, $\Ab$, $\tore$ et 
$\res$
. On peut donc appliquer le lemme \ref{l3.0} \`a toutes ces th\'eories motiviques. 
\end{proof}

\begin{thm}\phantomsection \label{t4.3} a) $f_*$ existe pour $\Cor_\sim(-,A)$, $\Corf$, $\Mot_\sim^\eff(-,A)$,
\allowbreak $\Mot_\sim(-,A)$, $\sS\sH^\eff$, $\sS\sH$, $\DM_\gm^\eff$, $\DM_\gm$, $\DM_-^\eff$, $\DM^\eff,\DM$; il co\"\i ncide avec $f_\sharp$.\\
b) $f_\sharp$ existe pour les th\'eories $\Grp$, $\SAb$, 
$\Ab$, $\tore$, $\res$
; il co\"\i ncide avec $f_*$. Les changements de base \eqref{eq1.1} entre ces diverses th\'eories sont des isomorphismes.
\end{thm}

\begin{proof} a)  Comme $\Mot_\sim(-,A)$ est munie d'une dualit\'e, on applique le lemme \ref{l3.5}; on obtient, pour $X$ projective lisse de dimension $d$ et $n\in \Z$: 
\[f_*(h(X)(n))=(f_\sharp h(X)(-d)(-n))^*= (h(f_\sharp X)(-d)(-n))^*=h(f_\sharp X)(n)\]
puisque $\dim f_\sharp X=\dim X$. En r\'eappliquant le lemme \ref{l3.0} ceci montre que $f_*$ pr\'eserve $\Mot_\sim^\eff$ et $\Cor_\sim$, et prend la m\^eme valeur que $f_\sharp$ sur les objets. De m\^eme, si $\phi:h(X)(m)\to h(Y)(n)$ est un morphisme, on a
\[f_*(\phi) = {}^t(f_\sharp({}^t\phi))=f_\sharp(\phi)\]
par un calcul direct.

Pour $\Corf$, on procède aussi par calcul direct, en notant qu'un fermé intègre $Z$ de $X\times_k Y_{(k)}=X_K\times_K Y$, pour $X\in \Corf(k)$ et $Y\in \Corf(K)$, est fini et surjectif sur $X$ si et seulement s'il l'est sur $X_K$. Ceci permet d'attraper d'abord $\DM_\gm^\eff$ et $\DM_\gm$, y compris l'isomorphisme $f_\sharp\simeq f_*$.

Dans le cas de $\sS\sH^\eff,\sS\sH,\DM^\eff$ et $\DM$, le théorème de représentabilité de Brown à la Neeman implique l'existence de $f_*$.   Pour obtenir l'isomorphisme $f_\sharp\simeq f_*$ sur ces grosses catégories, 
il suffit alors de vérifier que $f_*$ commute aux sommes directes infinies, ce qui résulte formellement du fait que $f^*$ préserve les objects compacts.

b) Toutes ces cat\'egories sont des sous-cat\'egories pleines de celle des faisceaux de groupes ab\'eliens pour la topologie \'etale; l'\'enonc\'e r\'esulte alors de \cite[Ch. V, Lemma 1.12]{milne2}. La derni\`ere affirmation r\'esulte du lemme \ref{l3.0}.
\end{proof}

\begin{rque} Le cas de $\sH$ n'est pas traité ci-dessus par manque de référence. Une variante (à dégager) du théorème de representabilité de Brown devrait le donner de la même façon.
\end{rque}

\subsection{Extensions radicielles} Ici, la situation est bien plus simple:

\begin{prop}\phantomsection\label{l3.1} Soit $K/k$ une extension radicielle de corps de caractéristique $p>0$. Alors, pour toute $\Z[1/p]$-algèbre commutative  $A$, le foncteur d'extension des scalaires
\[\Corf(k)\otimes A\to \Corf(K)\otimes A\]
est une \'equivalence de cat\'egories. Ceci s'étend aux théories $\DM_\gm^\eff$, $\DM_\gm$, $\DM^\eff$ et $\DM$.
\end{prop}

\begin{proof} Dans le cas de $\Corf$, la même que celle de \cite[prop. 1.7.2]{birat-pure}. On passe de là à $K^b(\Corf)$, puis à son quotient par les relations (Mayer-Vietoris + $\A^1$-invariance), puis à l'enveloppe pseudo-abélienne, puis en inversant l'objet de Tate.

Pour les grosses catégories, on commence par traiter le cas des préfaisceaux avec transferts $\PST$ (\ref{2.4.1}): notons que le cas de $\Corf$ implique celui de $\PST$ (cf. \cite[lemma 6.5.1]{birat}). On passe ensuite aux faisceaux Nisnevich avec transferts, etc.
\end{proof}

\begin{rque}\phantomsection Ce résultat a été annoncé par Suslin en 2007 dans un exposé au Fields Institute de Toronto. Il est paru récemment; la démonstration est essentiellement la même \cite{suslin}.
\end{rque}

\subsection{Cas g\'en\'eral}

\begin{thm}\phantomsection\label{t3.1} Soit  $K/k$ une extension finie, et soit $A$ une $\Z[1/p]$-algèbre commutative, où $p$ est l'exposant caractéristique de $k$. Alors le foncteur
\[\sM(k)\to \sM(K)\]
admet un adjoint \`a gauche et un adjoint \`a droite qui sont canoniquement isomorphes, pour les théories $\sM$ suivantes: $\Mot_\sim(-,A)$, $\Mot_\sim^\eff(-,A)$, $\Corf\otimes A$, $\DM_\gm^\eff\otimes A$, $\DM_\gm\otimes A$, $\DM^\eff(-,A)$ et $\DM(-,A)$. Les morphismes de changement de base entre ces diverses théories sont des isomorphismes.
\end{thm}

\begin{proof} Dans le cas des motifs purs, on empile \cite[prop. 1.7.2]{birat-pure} sur les théorèmes \ref{t4.1} et \ref{t4.3}. De même pour les autres théories, en utilisant la proposition \ref{l3.1}.
\end{proof}

\subsection{Application: un théorème d'effectivité}

\begin{thm}\phantomsection\label{t2} Soit $(A,\sim)$ un couple adéquat, où $A$ est une $\Q$-algèbre et où $\sim$ est l'équivalence rationnelle, algébrique, smash-nilpotente, homologique (pour une cohomologie de Weil classique) ou numérique, et soit $f:k\to K$ une extension. Soit
$M\in\Mot_\sim(k,A)$. Si $f^*M\in \Mot_\sim(K,A)$ est effectif, alors $M$ est effectif.
\end{thm}

\begin{proof} Sauf dans le cas de l'équivalence numérique, tout ce que la démonstration utilise est que $\sim$ se spécialise (cf. \cite[th. 6.8.3]{zetaL}). On se ramène d'abord à $K/k$ de type fini, puis aux deux cas essentiels:
\begin{thlist}
\item $K/k$ finie;
\item $K=k(t)$.
\end{thlist}

Dans le cas (i), la proposition \ref{p2bis} montre que  $M$ est facteur direct de $f_\sharp f^*M(=M\otimes f_\sharp \un)$, qui est effectif d'après le théorème \ref{t3.1}.

Dans le cas (ii), supposons d'abord que $\sim\ne \num$. Comme $f^*\tilde M$ provient de $k$, il a bonne réduction partout, par exemple en  $t=0$. En y spécialisant $f^*\tilde M$, on obtient un motif effectif, qui est évidemment $\tilde M$.

Il n'est pas clair que les cycles modulo l'équivalence numérique se spécialisent; pour traiter $\sim=\num$, choisissons une théorie motivique associée à une cohomologie de Weil classique $H$ (par exemple la cohomologie $l$-adique pour $l\ne \car k$). Comme le noyau de $\Mot_H(k)\to \Mot_\num(k)$ est localement nilpotent \cite[prop. 5]{aknote},  $M$ se rel\`eve en $\tilde M\in \Mot_H(k)$, et de plus $M$ effectif $\iff$ $\tilde M$ effectif. En appliquant cette remarque sur $k(t)$, on se ramène à montrer que $f^*\tilde M$ effectif $\Rightarrow$ $\tilde M$ effectif, ce qui vient d'être fait.
\end{proof}

\section{Motifs num\'eriques}\label{s5}

\subsection{Adjoints dans les cat\'egories semi-simples}

Soit $F$ un corps commutatif et soit $\sA$ une cat\'egorie $F$-lin\'eaire semi-simple
\cite[2.1.1, 2.1.2]{nrsm}. On suppose $\sA$ \emph{pseudo-ab\'elienne} (donc ab\'elienne, loc.
cit.).

\begin{defn}\phantomsection Soit $\Phi:\sA\to \Vec_F$ un foncteur $F$-lin\'eaire de $\sA$ vers la cat\'egorie
des $F$-espaces vectoriels (un $\sA$-module \`a gauche au sens de \cite[1.3.3]{nrsm}). Le
\emph{support} de
$\Phi$ est
\[\Supp(\Phi)=\{[S]\mid S \text{ simple}, \Phi(S)\ne 0\}
\]
o\`u $[M]$ d\'esigne la classe d'isomorphisme d'un objet $M\in \sA$.
\end{defn}

\begin{lemme}\phantomsection\label{l1} Les conditions suivantes sont \'equivalentes:
\begin{thlist}
\item $\Phi$ est corepr\'esentable;
\item $\Supp(\Phi)$ est fini et, pour tout $S$ simple, $\dim_F \Phi(S)<\infty$.
\end{thlist}
\end{lemme}

\begin{proof} (i) $\Rightarrow$ (ii): si $\Phi(M)=\sA(N,M)$ pour un $N\in \sA$, alors pour $S$
simple, on a $\Phi(S)\ne 0$ si et seulement si $S$ est facteur direct de $N$ (lemme de Schur).
La seconde condition est claire (de nouveau, lemme de Schur).

(ii) $\Rightarrow$ (i): pour tout $S$ simple, $\Phi(S)$ est un $\End_\sA(S)$-module \`a gauche:
soit $n_S$ sa dimension (qui est suppos\'ee finie). Posons 
\[N=\bigoplus_{[T]} T^{n_T}\]
o\`u on a choisi, pour chaque classe d'isomorphisme $[T]$ d'objets simples, un repr\'esentant
$T\in [T]$. Par construction, on a pour tout $S$ simple un isomorphisme de
$\End_\sA(S)$-modules:
\[\Phi(S)\simeq \sA(N,S).\]

Comme $\sA$ est semi-simple, ces isomorphismes d\'efinissent un isomorphisme de foncteurs
$\Phi\simeq \sA(N,-)$.
\end{proof}

\begin{prop}\phantomsection\label{p1} Soit $F$ un corps, et soit $T:\sA\to \sB$ un foncteur $F$-lin\'eaire
pleinement fid\`ele entre deux
$F$-cat\'egories ab\'eliennes semi-simples.\\
a) $T$ admet un adjoint \`a gauche
$T_\sharp$ et un adjoint \`a droite $T_*$.\\
b) Pour tout $S\in \sB$ simple, on a
\[T_\sharp S=
\begin{cases}
0 &\text{si $S$ n'est pas de la forme $T(S')$, $S'\in \sA$;}\\
$S'$ &\text{si $S$ est de la forme $T(S')$, $S'\in \sA$.}
\end{cases}
\]
En particulier, $T_\sharp(S)$ est nul ou simple. De m\^eme pour $T_*(S)$.\\
c) Le morphisme canonique de foncteurs
\[T_*\Rightarrow T_\sharp\]
(d\'eduit par pleine fid\'elit\'e de $T$ de la composition 
\[TT_*\Rightarrow Id_\sB \Rightarrow TT_\sharp\]
de la co\"unit\'e de $T_*$ et de l'unit\'e de $T_!$) est un isomorphisme.
\end{prop}

\begin{proof} a) Traitons le cas d'un adjoint \`a gauche (celui d'un adjoint \`a droite est
dual). Soit $N\in \sB$: il faut voir que le foncteur
\[\Phi(M)=\sB(N,T(M))\]
est repr\'esentable. On peut \'evidemment supposer $N$ \emph{simple}.

Appliquons le crit\`ere du lemme \ref{l1}: on doit v\'erifier que, pour tout $S\in \sA$ simple,
$\dim_F \Phi(S)$ est finie et nulle pour presque tout $S$. Or, comme $T$ est pleinement
fid\`ele, $T(S)$ est simple (lemme de Schur), donc $\Phi(S)\ne 0$ $\iff$ $T(S)\simeq N$, et dans
ce cas on a $\Phi(S)\simeq \End_\sB(N)$ qui est \'evidemment de dimension finie.

b) Cela r\'esulte de la formule donnant $T_\sharp S$ ou $T_* S$, cf. preuve du lemme \ref{l1}.

c) Il suffit de montrer que la
composition
\[TT_*S\to S\to TT_\sharp S\]
est un isomorphisme pour tout $S\in \sB$ simple. Or d'apr\`es b), les deux termes extr\^emes
sont nuls si $S$ n'est pas d\'efini sur $\sA$, et les deux fl\`eches sont des isomorphismes si
$S$ est d\'efini sur $\sA$.
\end{proof}

\begin{rque}\phantomsection Cet argument s'étend au cas où $T$ est seulement plein: $T(S)$ est alors simple ou nul. Je laisse au lecteur le soin de décrire $T_\sharp$ et $T_*$ explicitement dans ce cas.
\end{rque}

\subsection{Un r\'esultat de pleine fid\'elit\'e}

\begin{prop}\phantomsection\label{p5} Soit $(A,\sim)$ un couple ad\'equat, o\`u $A\supset \Q$ et $\sim$ est 
moins fine que l'\'equivalence alg\'ebrique. Alors, pour toute extension
\emph{primaire}
$f:k\to K$ (extension radicielle d'une extension r\'eguli\`ere), les foncteurs d'extension des
scalaires
\begin{align*}
f^*:\Cor_\sim(k,A)&\to \Cor_\sim(K,A)\\
f^*:\Mot_\sim^\eff(k,A)&\to \Mot_\sim^\eff(K,A)\\
f^*:\Mot_\sim(k,A)&\to \Mot_\sim(K,A)
\end{align*}
sont pleinement fid\`eles.
\end{prop}

\begin{proof} Il suffit de la faire pour le premier foncteur. Pour cela, on doit 
montrer que, pour deux $k$-vari\'et\'es projectives lisses $X,Y$, le morphisme
\[
\sZ^d_\sim(X\times Y,A)\to \sZ^d_\sim((X\times Y)_K,A)
\]
est bijectif, o\`u $d=\dim Y$. La proposition 1.7.2 de \cite{birat-pure} nous ram\`ene au cas o\`u $K/k$ 
est r\'eguli\`ere. Soit $K_s$ une cl\^oture s\'eparable de $K$, $k_s$ la fermeture 
s\'eparable de $k$ dans $K_s$, $G_K=Gal(K_s/K)$ et $G_k=Gal(k_s/k)$. On a un diagramme commutatif
\[\begin{CD}
\sZ^d_\sim((X\times Y)_K,A)@>\sim>> \sZ^d_\sim((X\times Y)_{K_s},A)^{G_K}\\
@AAA @AAA\\
\sZ^d_\sim(X\times Y,A)@>\sim>> \sZ^d_\sim((X\times Y)_{k_s},A)^{G_k}
\end{CD}\]
o\`u les fl\`eches horizontales sont des isomorphismes. De plus, il est bien connu que la fl\`eche
\[\sZ^d_\sim((X\times Y)_{k_s},A)\to \sZ^d_\sim((X\times Y)_{K_s},A)\]
est  un  isomorphisme, puisque $\sim$ est moins fine que l'\'equivalence alg\'ebrique \cite[p. 448]{kleiman}. Comme $K/k$ 
est r\'eguli\`ere, $G_K\to G_k$ est surjectif, donc la fl\`eche verticale de droite du diagramme 
est bijective, ainsi que celle de gauche.
\end{proof}

\subsection{Motifs num\'eriques: existence de l'adjoint} \`A partir de maintenant, $\sim=\num$ (\'equivalence num\'erique) et  $A=\Q$; on l'omet des notations.

\begin{thm}\phantomsection\label{t1} Pour toute
extension primaire
$f:k\to K$, le foncteur $f^*:\Mot_\num(k)\allowbreak\to \Mot_\num(K)$
admet un adjoint \`a gauche $f_\sharp$ et un adjoint \`a droite $f_*$; ces adjoints
sont ca\-no\-ni\-que\-ment isomorphes.\\
Le m\^eme \'enonc\'e vaut en rempla\c cant $\Mot_\num$ par $\Mot_\num^\eff$ (motifs
effectifs).
\end{thm}

\begin{proof} R\'esulte du th\'eor\`eme de semi-simplicit\'e de Jannsen
\cite{jannsen} et des propositions \ref{p1} et \ref{p5}.
\end{proof}

\subsection{Compatibilit\'e aux twists}

\begin{prop}\phantomsection\label{p2} a) Pour tout $(M,N)\in \Mot_\num(k)\times \Mot_\num(K)$, on a la
formule de projection
\[f_\sharp(N\otimes f^*M)\simeq f_\sharp N\otimes M.\]
b) Pour tout $M\in \Mot_\num(K)$ et pour tout $n\in \Z$, on a un isomorphisme
\[f_\sharp(M(n))\simeq (f_\sharp M)(n).\]
\end{prop}

\begin{proof} a) Cela r\'esulte de la proposition \ref{p2bis}. b) C'est le cas particulier 
$N=\bL^{\otimes n}$ o\`u $\bL$ est le motif de Lefschetz.
\end{proof}

\subsection{Compatibilit\'e aux poids} Pour \'enoncer le r\'esultat, nous avons besoin de nous
restreindre \`a une sous-th\'eorie pleine convenable de $\Mot_\num$ (conjecturalement \'egale
\`a $\Mot_\num$):

Pour une des th\'eories motiviques homologiques $\hom$ de \ref{s2.2.2}, notons comme dans
\cite[rem. 1]{aknote} $\Mot_\num^*(k)$ la sous-cat\'egorie pleine de $\Mot_\num(k)$ form\'e des
objets dont un relev\'e dans $\Mot_\hom(k)$ a tous ses projecteurs de K\"unneth 
alg\'ebriques. D'apr\`es \cite[prop. 6]{aknote}, $\Mot_\num^*(k)$ ne d\'epend pas du choix de
la th\'eorie homologique de \ref{s2.2.2}. (La preuve de cette proposition, incompl\`ete en
caract\'eristique
$>0$  dans \cite{aknote}, est compl\'et\'ee dans \cite[app. B]{nrsm}.) Les projecteurs de
K\"unneth d\'efinissent une $\Z$-graduation sur le foncteur identique de $\Mot_\num^*$. Il est
clair que l'extension des scalaires respecte $\Mot_\num^*$ et la graduation ci-dessus.

\begin{prop}\phantomsection\label{p5.1} Soit $f:k\to K$ une extension primaire. Alors $f_\sharp\Mot_\num^*(K)\subset
\Mot_\num^*(k)$, et $f_\sharp$ respecte les $\Z$-graduations.
\end{prop}

(Comme $f^*$ est pleinement fid\`ele, $f_\sharp f^*\simeq Id$ et l'inclusion est une
\'egalit\'e.)

\begin{proof} Il est \'equivalent de montrer que, si $M\in \Mot_\num(K)$ est pur de poids $w$,
alors $f_\sharp M$ est pur de poids $w$. On peut supposer $M$ simple. Si $f_\sharp M =0$, il n'y
a rien \`a d\'emontrer. Sinon, d'apr\`es la preuve de la proposition \ref{p1}, $M\simeq f^* S$
avec $S$ simple, et $f_\sharp M\simeq S$. Soit $\tilde S\in \Mot_\hom(k)$ un objet s'envoyant
sur $S$. Alors $f^*\tilde S\in \Mot_\hom(K)$ s'envoie sur $M$, et
\[H^*(f^*\tilde S)\simeq H^*(\tilde S)\otimes_{F(k)} F(K).\]

Ceci montre que $H^i(\tilde S)=0$ pour $i\ne w$, donc, par d\'efinition, que $S = f_\sharp M$
est pur de poids $w$.
\end{proof}

\begin{ex}\phantomsection\label{ex5.1} Via le foncteur $A\mapsto h_1(A)$, les adjoints $f_\sharp$ et $f_*$ \'eten\-dent 
respectivement la $K/k$-image et la $K/k$-trace de Chow sur les vari\'et\'es 
ab\'eliennes, \`a isog\'enie
pr\`es. Ceci r\'esulte du lemme \ref{l3.0} et du fait que 
ce foncteur de $\Ab\otimes\Q$ vers $\Mot_\num$ est pleinement fid\`ele. L'isomorphisme 
entre $f_*$ et $f_\sharp$ redonne l'isog\'enie entre la $K/k$-trace et la $K/k$-image.
\end{ex}

\subsection{Effectivit\'e; motifs nu\-m\'e\-ri\-ques birationnels}

\begin{thm}\phantomsection Le morphisme de changement de base \eqref{eq1.1} associé au morphisme de théories motiviques $\Mot_\num^\eff\to \Mot_\num$ est un isomorphisme 
pour toute extension primaire $K/k$. 
\end{thm}

\begin{proof} Soit $S\in \Mot_\num^\eff(K)$, simple. Si $S$ n'est pas d\'efini sur $k$ en tant
que motif non n\'ecessairement effectif, il ne l'est a fortiori pas en tant que motif effectif,
et $f_* S=f_\sharp S=0$ dans les deux cat\'egories. Si $S$ est d\'efini sur $k$ par $S_0$ dans
$\Mot_\num$, alors $S_0\in\Mot_\num^\eff$ d'apr\`es le th\'eor\`eme \ref{t2}, donc $f_*
S=f_\sharp S=S_0$ dans $\Mot_\num^\eff(k)$ et dans $\Mot_\num(k)$.
\end{proof}

\begin{cor}\phantomsection\label{c1} La proposition \ref{p2} reste valable en rempla\c cant la cat\'egorie $\Mot_\num$ par
$\Mot_\num^\eff$ (dans b), prendre $n\ge 0$).\qed
\end{cor}

Dans le corollaire qui suit, on utilise la théorie $\Mot^\o$ des motifs numériques birationnels déjà utilisée dans \cite[Part II]{qjpam} et \cite{birat-pure} (cf. \S \ref{s.bir}).

\begin{cor}\phantomsection\label{c2} Pour tout corps $k$, soit $i_k:\Mot_\num^\o(k)\to \Mot_\num^\eff(k)$
l'adjoint
\`a gauche et \`a droite de la projection $\Mot_\num^\eff(k)\to \Mot_\num^\o(k)$ (cf.
\cite[prop. 7.7 b)]{qjpam}). Alors, pour toute extension primaire $f:k\to K$,\\
a)  le diagramme
\[\begin{CD}
\Mot_\num^\o(K)@>i_K>> \Mot_\num^\eff(K)\\
@A{f^{*\o}}AA @A{f^{*\eff}}AA\\
\Mot_\num^\o(k)@>i_k>> \Mot_\num^\eff(k)
\end{CD}\]
est naturellement commutatif. (Autrement dit, $i$ d\'efinit un \emph{morphisme de th\'eories
motiviques} au sens de \ref{s1.3}.) \\ 
b) Le foncteur $f^{*\o}$ est pleinement
fid\`ele et admet un adjoint \`a gauche et un adjoint \`a droite, canoniquement isomorphes.\\
c) Le morphisme de changement de base relatif au morphisme de projection $\pi:\Mot^\eff_\num\to
\Mot^\o_\num$ est un isomorphisme.\\
d) De m\^eme pour celui relatif \`a son adjoint $i$.\\
e) La formule de projection de la proposition \ref{p2} a) reste valable pour $\Mot_\num^\o$.
\end{cor}

\begin{proof} a) Soit $\bL$ le motif de Lefschetz. Rappelons \cite[prop. 7.7 c)]{qjpam} que pour
tout corps
$k$, $\Mot_\num^\eff(k)$ est le coproduit
de ses sous-cat\'egories pleines $\Mot_\num^\eff(k)\otimes \bL$ et $i_k(\Mot_\num^\o(k))$, i.e.
tout objet s'\'ecrit de mani\`ere unique comme somme directe d'objets de ces deux
sous-cat\'egories. En particulier, $\Mot_\num^\o(k)$ est ab\'elienne semi-sim\-ple et $i_k$
est pleinement fid\`ele, d'image engendr\'ee par les objets simples
$S\in \Mot_\num^\eff(k)$ tels que $S(-1)$ ne soit pas effectif. On appelle ces derniers motifs
\emph{primitifs}.

Le foncteur $f^{*\eff}$ envoie $\Mot_\num^\eff(k)\otimes \bL$ dans $\Mot_\num^\eff(K)\otimes
\bL$, et le point est de montrer qu'il envoie aussi $i_k(\Mot_\num^\o(k))$ dans
$i_K(\Mot_\num^\o(K))$.

Soit donc $S\in\Mot_\num^\o(k)$ un objet simple.  Supposons que le motif
$f^{*\eff}(i_k S)(-1)=f^{*\eff}(i_k S(-1))$ soit effectif. D'apr\`es le th\'eor\`eme \ref{t2},
$(i_k S)(-1)$ est effectif, ce qui est absurde.

Comme $K/k$ est primaire, $f^{*\eff}$ est pleinement fid\`ele (th\'eor\`eme \ref{t1}). En
particulier, $f^{*\eff}(i_k S)$ est simple.  Par cons\'equent,
$f^{*\eff}(i_k S)\simeq i_K S'$ pour un $S'\in \Mot_\num^\o(K)$ simple. 

Dans b), la pleine fid\'elit\'e r\'esulte de a). Le reste r\'esulte de la proposition
\ref{p1}. 

c) Il suffit de tester sur un objet simple 
$S\in \Mot_\num^\eff(K)$. Si $S\simeq S'(1)$ (avec $S'$ effectif), alors $f_\sharp^\eff S\simeq
(f_\sharp^\eff S')(1)$ par le corollaire \ref{c1}, et dans le morphisme
\[f_\sharp^\o\pi_K S\to \pi_k f_\sharp^\eff S\]
les deux membres sont nuls. Supposons maintenant $S$ primitif. D'apr\`es a), $S$ est 
d\'efini sur
$k$ si et seulement si $\pi_K S$ est d\'efini sur $k$. Si ce n'est pas le cas, les deux membres
sont encore nuls. Si $S\simeq f^{*\eff} S'$, alors $\pi_K S\simeq f^{* o} \pi_k S'$, les deux
membres s'identifient \`a $\pi_k S'$ et la fl\`eche \`a l'identit\'e. 

d) R\'esulte soit de la d\'emonstration de c), soit de la compatibilit\'e des
adjoints \`a la composition.

e) Soit $(M,N)\in \Mot_\num^\o(K)\times \Mot_\num^\o(k)$. D'apr\`es le corollaire \ref{c1}, on a
un isomorphisme dans $\Mot_\num^\eff(k)$
\[f_\sharp^\eff(i_K M\otimes f^{*\eff} i_k N)\simeq (f_\sharp^\eff i_K M)\otimes i_k N.\]

On obtient la formule voulue en appliquant $\pi_k$ \`a cet isomorphisme, en utilisant que
$\pi_k$ commute au produit tensoriel et en utilisant c).
\end{proof}

\subsection{Motifs homologiques} \label{5.7} Conjecturalement, équivalences homologique et numérique coïncident; les adjoints à gauche et à droite du foncteur d'extension des scalaires devraient donc exister pour les motifs homologiques. Voici tout ce que je sais dire:

\begin{prop}\phantomsection\label{p5.2} Soient $p=0$ ou un nombre premier, $l\ne p$ un nombre premier, et $F_l$ la théorie motivique de la définition \ref{d2.1}. Alors, pour toute extension $f:k\to K$ de type fini dans $\corps_p$, le foncteur d'``extension des scalaires'' $f^*:F_l(k)\to F_l(K)$ admet un adjoint à gauche $f_\sharp$ et un adjoint à droite $f_*$. Ces adjoints sont compatibles aux poids et respectent la sous-théorie $F_l^{ss}$ de la remarque \ref{r2.1}; le morphisme canonique $f_*\Rightarrow f_\sharp$, restreint à $F_l^{ss}(K)$, est un isomorphisme.
\end{prop}

\begin{proof} On se ramène au cas où $k$ est de type fini sur son sous-corps premier, puis en utilisant le  théorème \ref{t3.1} au cas où $K/k$ est régulière. Soit $H=\Ker(G_K\to G_k)$: on voit tout de suite que $f_\sharp$ et $f_*$ sont donnés par
\[f_\sharp V = V_H, \quad f_* V = V^H.\]

La compatibilité aux poids est évidente, ainsi que le fait que ces foncteurs envoient $F_l^{ss}(K)$ dans $F_l^{ss}(k)$ (une sous-représentation et une représentation quotient d'une représentation semi-simple sont semi-simples).
\end{proof}

Malheureusement, la conjecture de Tate ne semble pas suffisante pour en déduire l'existence d'adjoints pour les motifs homologiques: outre le problème de descendre les coefficients de $\Q_l$ à $\Q$, se pose un problème d'essentielle surjectivité. Quand $k$ est un corps fini, on peut le résoudre grâce à Weil I et au théorème de Honda: en effet, les valeurs propres de Frobenius opérant sur $V^H$ sont des nombres de Weil, comme on le voit par un argument de spécialisation. Si $k$ est un corps de nombres, on peut peut-être utiliser la conjecture de Fontaine-Mazur (en plus de la conjecture de Tate!) pour conclure. Pour d'autres corps $k$, je ne sais pas. Peut-on remplacer ces rustines un peu fortes par, disons, la conjecture de Tate généralisée?

\section{Cat\'egories birationnelles}\label{s6}

\subsection{Cat\'egorie birationnelle des vari\'et\'es lisses}

\begin{thm}\phantomsection\label{t3} Soit  $f:k\to K$ une extension de type
fini séparable. Soit $S_b$ l'ensemble des morphismes
birationnels. Alors le foncteur
\[f^*:S_b^{-1}\Sm(k)\to S_b^{-1}\Sm(K)\]
admet un adjoint \`a gauche $f_\sharp$. Pour $X\in S_b^{-1}\Sm(K)$, on a $f_\sharp X = \sX$ pour tout $k$-modèle lisse $\sX$ de $X$.
\end{thm}

Nous allons donner deux d\'emonstrations. La premi\`ere utilise la r\'e\-so\-lu\-tion des
singularit\'es de Hironaka, et n'est donc valable qu'en caract\'eristique z\'ero; la seconde
est valable en toute caract\'eristique.

\begin{proof}[Premi\`ere d\'emonstration] D'après \cite[th. 6.6.3]{birat}, on a une bijection canonique
\begin{equation}\label{eq6.1}
S_b^{-1}\Sm(k)(V,W)\simeq W(k(V))/R
\end{equation}
pour tout couple $(V,W)$ de $k$-variétés lisses, avec $W$ propre, où le membre de droite est l'ensemble des classes de $R$-équivalence. Supposons maintenant $k$ de caractéristique $0$. D'apr\`es \cite[prop. 8.5]{localisation}, on peut remplacer les vari\'et\'es
lisses $\Sm$ par les vari\'et\'es projectives lisses $\Sm^\proj$. 
Soit $X\in S_b^{-1}\Sm^\proj(K)$. On veut montrer que le foncteur
\[Y\mapsto \Hom(X,Y_K)\]
de $S_b^{-1}\Sm^\proj(k)$ vers la cat\'egorie des ensembles est corepr\'esentable. D'apr\`es
\eqref{eq6.1}, la valeur de ce foncteur sur $Y$ est
\[Y_K(K(X))/R\]
o\`u $R$ d\'esigne la R-\'equivalence. Soit $\sX$ une $k$-vari\'et\'e projective lisse telle que $k(\sX)\simeq K(X)$ (r\'esolution des
singularit\'es de Hironaka). On a alors
\[Y_K(K(X))/R\simeq Y(k(\sX))/R\]
et notre foncteur est corepr\'esent\'e par $\sX$.
\end{proof}

\begin{proof}[Seconde d\'emonstration] Rappelons que ``séparable de type fini'' $\iff$ ``admet un modèle lisse'' \cite[lemma 4.8.1]{birat}. Pour tout modèle lisse $U$ de $K/k$,  soit $\Sm(U)$ la cat\'egorie des $U$-sch\'emas lisses s\'epar\'es de type fini. On a
\[2-\colim_U \Sm(U)\iso \Sm(K)\]
(limite prise sur les immersions ouvertes). Si $p_U:U\to \Spec k$ est le morphisme structural, on a un couple de foncteurs
adjoints:
\[p_U^*:\Sm(k)\leftrightarrows \Sm(U):(p_U)_\sharp\]
donn\'es par l'extension et la restriction des scalaires ($(p_U)_\sharp$ est adjoint à gauche de $p_U^*$). De même, si $j:U'\to U$ est une immersion ouverte, on a un couple de foncteurs adjoints
\[j^*:\Sm(U)\leftrightarrows \Sm(U'):j_\sharp.\]

Tous ces foncteurs respectent les morphismes birationnels. D'après la proposition \ref{p.2colim}, pour conclure il suffit de démontrer que les foncteurs
\begin{equation}\label{eq6.2}
S_b^{-1} j^*:S_b^{-1}\Sm(U)\to S_b^{-1}\Sm(U')
\end{equation}
induits par les immersions ouvertes $j$ sont des équivalences de catégories. Mais c'est évident: l'adjonction $(j_\sharp,j^*)$ induit une adjonction $(S_b^{-1}j_\sharp,S_b^{-1} j^*)$ et il suffit de montrer que son unité et sa coünité sont des isomorphismes naturels. Or l'unité $Id_{\Sm(U')}\Rightarrow j^*j_\sharp$ est déjà un isomorphisme naturel, et la coünité $j_\sharp j^*\Rightarrow Id_{\Sm(U)}$ est une immersion ouverte quand on l'évalue sur tout $X\in \Sm(U)$.  La dernière assertion du théorème est évidente par construction.
\end{proof}

Pour le corollaire suivant, soient $K/k$ une extension séparable de type fini et  $X_1,X_2$ deux $K$-variétés propres et lisses munies de $k$-modèles propres et lisses $\sX_1$ et $\sX_2$. Soit $\alpha:X_1\tto X_2$ une $K$-application rationnelle: rappelons que $\alpha$ définit un morphisme dans $S_b^{-1}\Sm(K)$, donc une application   $\alpha_*(L):X_1(L)/R\to X_2(L)/R$ pour toute extension séparable de type fini $L/K$ (\cite[prop. 10]{ct-s} en caractéristique zéro,  \cite[6.3 et cor. 6.6.6]{birat} en général).

\begin{cor}\phantomsection\label{c6.1}  Supposons que $\alpha_*(L)$ soit bijective pour tout $L/K$.  Alors $\alpha$ induit un isomorphisme $\sX_1(L)/R\iso \sX_2(L)/R$ pour toute extension séparable de type fini $L/k$.
\end{cor}

\begin{proof}  Vu \eqref{eq6.1} et Yoneda, l'hypothèse et la conclusion sont respectivement équivalentes à $\alpha:X_1\iso X_2$ dans $S_b^{-1}\Sm(K)$ (resp. $\sX_1\iso \sX_2$ dans $S_b^{-1}\Sm(k)$). Via le théorème \ref{t3}, le corollaire résulte alors du fait que l'image d'un isomorphisme par un foncteur est un isomorphisme.
\end{proof}

Comme cas particulier, notons l'analogue suivant d'un théorème de Graber, Harris et Starr \cite[cor. 1.3]{ghs}: disons qu'une $k$-variété propre et lisse $X$ est \emph{universellement $R$-triviale} si l'ensemble $X(L)/R$ est un singleton pour toute extension séparable $L/k$.

\begin{cor}\phantomsection\label{c6.1a} Soit $\pi:\sX\tto Y$ une application rationnelle dominante de $k$-variétés propres, lisses et connexes, dont la restriction à un ouvert de définition est \emph{propre} (cf. \cite[déf. 5.12]{debarre}). Supposons que la fibre générique de $\pi$ soit universellement $R$-triviale. Alors $\sX(L)/R\iso Y(L)/R$ pour toute extension $L/k$.
\end{cor}

\begin{proof} Soit $K=k(Y)$: on prend $\sX_1=\sX$ et $\sX_2=Y$ dans le corollaire \ref{c6.1}.
\end{proof}

\begin{rque}\phantomsection On pourrait penser utiliser les \'equivalences de ca\-t\'e\-go\-ries \eqref{eq6.2} pour
d\'efinir des morphismes de sp\'ecialisation, mais les choses ne sont pas si simples. Plus
pr\'ecis\'ement, soit $i:Z\inj U$ une immersion ferm\'ee. On a un morphisme de changement de
base
\[i^*:\Sm(U)\to \Sm(Z)\]
mais ce morphisme n'envoie pas $S_b$ dans $S_b$. Par exemple, soit $j:U-Z\inj U$ l'immersion
ouverte compl\'ementaire. Alors $i^*j$ est l'inclusion du vide dans $Z$. Pour avoir un bon
foncteur sur les localisations, il faut donc se restreindre aux morphismes birationnels ``en
bonne position par rapport \`a $i$'': par d\'efinition, un $U$-morphisme birationnel $f:X\to Y$
est en bonne position par rapport \`a $i$ si $i^* f$ est encore un morphisme birationnel.
\end{rque}

\subsection{Correspondances finies birationnelles}

\begin{thm}\phantomsection\label{t6.1} Soit  $K/k$ une extension de type
fini séparable. Soit $S_b$ l'ensemble des morphismes
birationnels. Alors le foncteur
\[S_b^{-1}\Corf(k)\to S_b^{-1}\Corf(K)\]
admet un adjoint \`a gauche. Ceci s'étend au cas d'une extension $K/k$ de type fini quelconque, si on inverse l'exposant caractéristique.
\end{thm}

\begin{proof} La proposition \ref{l3.1} nous ramène au cas séparable. On procède alors comme dans la seconde démonstration du théorème \ref{t3}. Pour tout  modèle lisse $U$ de $K/k$, on a la cat\'egorie $\Corf(U)$ des correspondances finies de base $U$: ses objets sont les $U$-schémas $X/U$ lisses séparés de type fini, et
\[\Corf(U)(X/U,Y/U) = c_{equi}(X\times_U Y/X,0)\]
(\cite[app. 1A]{mvw}, \cite[déf. 1.2]{ivorra}, \cite{deglise}). On a
\[2-\colim_U \Corf(U)\iso \Corf(K)\]
\cite[cor. 4.12]{ivorra}. On peut alors reprendre mot pour mot la démonstration ci-dessus.
\end{proof}

\subsection{Motifs de Chow birationnels}

\begin{thm}\phantomsection\label{t4} Soit $K/k$ une extension de type fini. Soit $A$ un anneau
commutatif dans lequel $p$ est inversible, où $p$ est l'exposant caractéristique de $k$. Alors le foncteur d'extension des scalaires $\Chow^\o(k,A)\to
\Chow^\o(K,A)$ admet un adjoint \`a gauche.
\end{thm}

\begin{proof} D'après \cite[th. 3.4.1]{birat-tri}, on a une équivalence de catégories
\[\Chow^\o(k,A)\iso (S_b^{-1} \Corf(k)\otimes A)^\natural\]
et de même pour $K$, où $()^\natural$ désigne l'enveloppe pseudo-abélienne. La conclusion résulte donc de la proposition \ref{l3.1} b) et du lemme \ref{l3.2} b).
\end{proof}

\begin{rque}\phantomsection Comme dans le cas du th\'eor\`eme \ref{t3}, on peut donner une autre d\'emonstration du théorème \ref{t4}, en utilisant la formule
\begin{equation}\label{eq6.3}
\Chow^\o(k,\Z)(h^\o(X),h^\o(Y))\simeq CH_0(Y_{k(X)})
\end{equation}
de \cite[lemme 2.3.7]{birat-pure} pour deux $k$-vari\'et\'es projectives lisses $X,Y$. En caractéristique $0$, cette démonstration utilise la résolution des singularités; en caractéristique $>0$, elle devient extrêmement compliquée et je ne sais la faire marcher qu'en recourant au théorème de de Jong équivariant \cite[th. 7.3]{dJ}; en particulier, il faut supposer que $A$ est une $\Q$-algèbre. La reproduire ici ne présente donc aucun intérêt.
\end{rque}

Voici un analogue du corollaire \ref{c6.1}, qui redonne et précise le résultat de Vial \cite[th. 1.3]{vial} et une application de Sebastian \cite[\S 3.2]{seb2}:

\begin{cor}\phantomsection\label{c6.2} a) Soient $K/k$ une extension séparable de type fini, $X_1,X_2$ deux $K$-variétés projectives lisses munies de $k$-modèles projectifs lisses $\sX_1,\sX_2$, et soit $A$ une $\Z[1/p]$-algèbre commutative. Soit $\alpha\in CH_{\dim X_1}(X_1\times_K X_2)$ une correspondance algébrique. Supposons que $\alpha_*:CH_0((X_1)_L)\otimes A\iso CH_0((X_2)_L)\otimes A$ pour toute extension séparable de type fini $L/K$. Alors $\alpha$ induit un isomorphisme $CH_0((\sX_1)_L)\otimes A\iso CH_0((\sX_2)_L)\otimes A$ pour toute extension de type fini $L/k$.\\
b) Soit $\pi:\sX_1\tto \sX_2$ un $k$-``quotient rationnel'' au sens de Campana \cite[ch. 5]{debarre} (= ``MRCC fibration'' de Koll\'ar \cite[IV.5, Def. 5.1]{kollar}), et soit $A=\Q$. Alors $\pi$ induit un isomorphisme de motifs birationnels
\[h^\o(\sX_1)\iso h^\o(\sX_2)\]
et un isomorphisme $CH_0((\sX_1)_L)\otimes \Q\iso CH_0((\sX_2)_L)\otimes \Q$ pour toute extension de type fini $L/k$. 
\end{cor}

\begin{proof} a) La même que celle du corollaire \ref{c6.1}, en utilisant \eqref{eq6.3} et le théorème \ref{t4} au lieu de \eqref{eq6.1} et du théorème \ref{t3}. b) Soient $K=k(\sX_2)$, $X_1$ la fibre générique de $\pi$, $X_2=\Spec K$ et $\alpha$ le graphe de l'application rationnelle $\pi$. Par définition du quotient rationnel, $X_1$ est rationnellement connexe par chaînes, donc $\alpha$ induit un isomorphisme  $CH_0((X_1)_L)\otimes \Q\iso \Q$ pour toute extension $L/K$. L'énoncé est donc un cas particulier de a).
\end{proof}

\enlargethispage*{20pt}

\subsection{Motifs birationnels triangul\'es}

\begin{thm}\phantomsection\label{t5} Soit $K/k$ une extension de type fini, et soit
$A$ un anneau commutatif. Si $K/k$ est séparable, le foncteur d'extension des scalaires
$\DM_\gm^\o(k,A)\to\DM_\gm^\o(K,A)$ admet un adjoint \`a gauche. En général, ceci reste vrai pourvu que l'exposant caractéristique de $k$ soit inversible dans $A$.
\end{thm}

\begin{proof} Lorsque $K/k$ est séparable, on proc\`ede comme dans la seconde preuve du th\'e\-o\-r\`e\-me \ref{t3} en utilisant pour tout modèle lisse $U$ de $K/k$ la catégorie triangulée 
$\DM_\gm^\eff(U,A)$ des motifs g\'eom\'etriques de base $U$ \cite[déf. 1.14]{ivorra}. On peut alors copier mot pour mot la démonstration citée\footnote{Plus précisément, l'inversibilité des morphismes d'adjonction apparaissant dans cette démonstration se teste sur des générateurs: les motifs des $U$-schémas lisses.}, le point clé 
\[2-\colim_U \DM_\gm^\eff(U,A)\iso \DM_\gm^\eff(K,A)\]
étant \cite[prop. 4.16]{ivorra}.
\end{proof}

\begin{rque}\phantomsection Cette démonstration est légèrement incorrecte pour la raison suivante. Dans \cite{voetri}, $\DM_\gm^\eff$ est défini pour les corps en utilisant les ``relations'' de $\A^1$-invariance et de Mayer-Vietoris pour les recouvrements par deux ouverts de Zariski. Dans \cite[déf. 1.11]{ivorra}, $\DM_\gm^\eff$ est défini pour les schémas de la même manière, mais en remplaçant les recouvrements de Zariski par les recouvrements de Nisnevich élémentaires. Les deux définitions coïncident sur un corps parfait par un théorème hautement non trivial de Voevodsky (cf. \cite[th. 4.4.1 (2)]{birat-tri}). Voici deux manières de contourner ce problème en caractéristique $p>0$:
\begin{itemize}
\item Utiliser la proposition \ref{l3.1} pour étendre la coïncidence précédente aux corps imparfaits, quitte à inverser $p$.
\item Définir des catégories $\DM_\gm^\eff(U)$ ``à la Zariski'' pour $U$ lisse, et vérifier que la démonstration ci-dessus marche dans ce cas, ce qui est immédiat.
\end{itemize}
Dans les deux cas, un sous-produit de la démonstration est une équivalence de catégories
\[S_b^{-1}\DM_\gm^\eff(U)\iso S_b^{-1}\DM_\gm^\eff(K)\]
pour tout modèle lisse $U$ de $K/k$.
\end{rque}

\subsection{Comparaison des adjoints} Soit $A$ une $\Z[1/p]$-algèbre commutative, où $p=1$ ou un nombre premier. Pour tout corps $k$ d'exposant caractéristique $p$, on a des foncteurs \cite[cor. 3.4.2 et th. 4.2.2]{birat-tri}
\[S_b^{-1}\Sm(k)\to \Chow^\o(k,A)\to \DM_\gm^\o(k,A)\]
qui définissent des morphismes de théories motiviques (sur $\corps_p$ si $p\ne 1$, sur $\corps_0$ si $p=1$)
\begin{equation}\label{eq6.4}
S_b^{-1}\Sm \to \Chow^\o(-,A)\to \DM_\gm^\o(-,A).
\end{equation}

\begin{prop}\phantomsection Les morphismes de changement de base associés à \eqref{eq6.4} via les théorèmes \ref{t3}, \ref{t4} et \ref{t5} sont des isomorphismes pour toute extension séparable $K/k$.
\end{prop}

\begin{proof} \'Etant donné les démonstrations des trois théorèmes cités, il suffit de le vérifier pour les adjoints relatifs à un modèle lisse $U$ de $K/k$ (en remplaçant $\Chow^\eff$ par $\Corf$). Dans le cas de $S_b^{-1}\Sm$ et de $\Corf$, l'adjoint à gauche envoie un $U$-schéma-lisse $X$ sur lui-même; c'est encore le cas pour le motif de $X$ dans $\DM_\gm^\eff(U)$.
\end{proof}

\enlargethispage*{30pt}

\section{Non existence d'adjoints}\label{s7}

Apr\`es avoir montr\'e l'existence d'adjoints dans des situations vari\'ees, il peut \^etre
instructif de donner des exemples de non existence. En voici de deux sortes, avec des
cat\'egories de motifs purs:

\subsection{Motifs purs num\'eriques, extension alg\'ebrique s\'eparable infinie} On a: 

\begin{lemme}Si $f:k\to K$
est une telle extension, l'adjoint \`a gauche $f_\sharp$ de $f^*:\Mot_\num(k)\to \Mot_\num(K)$
n'est pas d\'efini en $\un\in \Mot_\num(K)$. 
\end{lemme}

\begin{proof} Pour simplifier, supposons que $K$ soit la cl\^oture s\'eparable de $k$ (on pourrait
g\'en\'eraliser l'argument). Il s'agit de voir que le foncteur
\[\Mot_\num(k)\ni M\mapsto \Mot_\num(K)(\un,f^* M)\]
n'est pas repr\'esentable. Pour $M=h(X)$ ($X$ projective lisse), le terme de droite est
$A_0^\num(X_K)$. Pour $X$ de dimension z\'ero, soit $X = \Spec E$ avec $E$ \'etale sur $k$, cet
espace vectoriel est de dimension $[E:k]$. Si le foncteur \'etait repr\'esentable par un objet
$N$, on aurait donc, pour tout tel $E$
\[\dim \Mot_\num(k)(N,h(\Spec E)) = [E:k].\]

Or $\Mot_\num(k)(N,h(\Spec E))=\Mot_\num(k)(h_0(N),h(\Spec E))$. Si on \'ecrit $h_0(N)$ comme
facteur direct de $h(\Spec F)$ o\`u $F$ est une $k$-alg\`ebre \'etale, cet espace vectoriel est
facteur direct de
\[\Mot_\num(k)(h(\Spec F),h(\Spec E))=\Mot_\num(k)(\un,h(\Spec (E\otimes_k F))\]
et le membre de droite est de dimension \'egale au nombre de facteurs de l'alg\`ebre \'etale
$E\otimes_k F$. 

Choisissons pour $E$ une extension galoisienne de $k$ qui contient tous les facteurs de $F$. On
a alors $E\otimes_k F \simeq E^{[F:k]}$, donc 
\begin{multline*}
[E:k]=\dim \Mot_\num(k)(N,h(\Spec E))\\
\le \dim\Mot_\num(k)(h(\Spec F),h(\Spec E))=[F:k]
\end{multline*} 
donc $[E:k]$ serait born\'e, ce qui est absurde.
\end{proof}

\subsection{Motifs de Chow effectifs, extension r\'eguli\`ere infinie} 
 Commen\c cons par deux lemmes:
 
 \begin{lemme}\phantomsection\label{l10.1} a) Soit $A$ une alg\`ebre semi-primaire sur un corps $F$: le
radical $R$ de $A$ est nilpotent et $A/R$ est semi-simple. Soit $(e_i)$ une famille
d'idempotents de $A$, orthogonaux et de somme $1$, telle que leurs images $\bar e_i\in A/R$
soient \emph{centrales}. Soit d'autre part $\epsilon$ un idempotent de $A$. Alors il existe une
d\'ecomposition $\epsilon = \sum_i \epsilon_i$ en somme d'idempotents orthogonaux et un
\'el\'ement $r\in R$ tel que, pour tout $i$, $\epsilon_i$ divise $(1+r)e_i(1+r)^{-1}$
(c'est-\`a-dire que $(1+r)e_i(1+r)^{-1}-\epsilon_i$ est idempotent).\\
 b) Soit $\sA$ une cat\'egorie $F$-lin\'eaire pseudo-ab\'elienne, o\`u $F$ est un corps. Soit
$M\in \sA$, tel que $A=\End(M)$ soit une $F$-alg\`ebre primaire. Soit $M=\bigoplus M_i$ une
d\'ecomposition de $M$ en somme directe, et soit $N$ un facteur direct de $M$. Si les
idempotents de $A$ de noyaux $M_i$ sont \emph{centraux modulo $R$} (radical de $A$), alors il
existe une d\'ecomposition $N=\bigoplus N_i$, o\`u chaque $N_i$ est isomorphe \`a un facteur
direct de
$M_i$.
  \end{lemme}
 
  \begin{proof} a) Soient $B=\epsilon A\epsilon$ et $C=(1-\epsilon) A(1-\epsilon)$. Les
$F$-alg\`ebres $B$ et $C$ sont primaires, d'unit\'es respectives $\epsilon$ et $(1-\epsilon)$
et de radicaux respectifs $\epsilon R\epsilon$ et $(1-\epsilon) R(1-\epsilon)$. Si $\bar
\epsilon$ est l'image de $\epsilon$ dans $A/R$, les $\bar\epsilon \bar e_i$ sont des
idempotents orthogonaux de somme $\bar\epsilon$ dans $B/\epsilon R\epsilon$, et les
$(1-\bar\epsilon) \bar e_i$ sont des idempotents orthogonaux de somme $1-\bar\epsilon$ dans
$C/(1-\epsilon) R(1-\epsilon)$. Relevons-les respectivement en $(\epsilon_i)$ et $(\eta_i)$,
syst\`emes d'idempotents orthogonaux de somme $\epsilon$ et $1-\epsilon$ dans $B$ et $C$
\cite[lemma 5.4]{jannsen2}. On a $\epsilon_i\eta_j = \eta_j\epsilon_i =0$ pour tout $(i,j)$,
donc les $\epsilon_i+\eta_i$ forment un syst\`eme d'idempotents orthogonaux de somme $1$,
relevant les $\bar e_i$. En r\'eappliquant \cite[lemma 5.4]{jannsen2}, on voit qu'il existe
donc $n\in R$ tel que $\epsilon_i+\eta_i=(1+r)e_i(1+r)^{-1}$ pour tout $i$.
 
 b) C'est une simple traduction de a).
  \end{proof}

\begin{lemme}\phantomsection\label{l7.2} Soit $k$ un corps alg\'ebriquement clos, et soit $A\ne 0$ une vari\'et\'e
ab\'elienne sur $k$. Si $k$ n'est pas la cl\^oture alg\'ebrique d'un corps
fini, alors  $A(k)\otimes\Q \ne 0$.
\end{lemme}

\begin{proof} Ce lemme bien connu a un grand nombre de d\'e\-mons\-tra\-tions non triviales. 
L'une utilise les hauteurs, et une autre le th\'eor\`eme de sp\'ecialisation de N\'eron: 
voir \cite[appendice]{roy-wald} pour le cas où $k=\bar \Q$ (les arguments passent au cas o\`u $k$ est la clôture algébrique d'un corps $k_0(C)$ pour $C$ une courbe). On peut aussi raisonner en utilisant la partie facile 
du crit\`ere de N\'eron-Ogg-\v Safarevi\v c, etc.
\end{proof}

Soit maintenant $f:k\to K$ une extension r\'eguli\`ere infinie:  Nous allons montrer que
l'adjoint \`a gauche $f_\sharp$ de $f^*:\Chow(k,\Q)\to \Chow(K,\Q)$ n'existe pas en
g\'en\'eral.  Il suffit de montrer:

\begin{thm}\phantomsection\label{p10.1} Soient $k$ un corps alg\'ebriquement clos, $C$ une cour\-be projective
lisse de genre $g>0$ sur $k$ et $K=k(C)$. Supposons que $k$ ne soit pas la cl\^oture
alg\'ebrique d'un corps fini. Alors $f_\sharp$ n'est pas d\'efini en $\un\in
\Chow(K,\Q)$, ni en $\un\in\Chow^\eff(K,\Q)$.
\end{thm}

(Pour
attraper le cas de genre z\'ero, on peut ensuite projeter $C$ sur $\P^1$; si on note
$k\by{f}k(\P^1)\by{g}k(C)$ les extensions correspondantes, $f_\sharp$ n'est pas d\'efini en
$g_\sharp \un$, autrement sa valeur d\'efinirait $(gf)_\sharp\un$.)

\begin{proof} Pour simplifier, laissons tomber les coefficients $\Q$. Pour a), il s'agit de voir que le foncteur
\[\Chow(k)\ni M\mapsto \Chow(K)(\un,f^* M)\]
n'est pas corepr\'esentable. Soit $n\ge 0$. Pour $M=h(X)(-n)$ ($X$ projective lisse), le terme de droite est
$CH_n(X_K)\otimes\Q$. Si le foncteur \'etait corepr\'esentable par un objet
$N$, on aurait donc un isomorphisme fonctoriel en $X$
\[CH_n(X_K)\otimes\Q\simeq \Chow(k)(N,h(X)(-n)).\]

Soit $\phi\in\Chow(K)(\un,h(C_K))=CH_0(C_K)$ donn\'e par le point g\'e\-n\'e\-ri\-que
de $C$. Ce morphisme donne par adjonction un morphisme
\[\tilde\phi: N\to h(C)\]
qui induit la surjection fonctorielle en $X$
\begin{multline*}\Chow(k)(h(C),h(X)(-n))\simeq CH_{n+1}(C\times X)\otimes\Q\\
\surj CH_n(X_K)\otimes\Q\simeq \Chow(k)(N,h(X)(-n)).
\end{multline*}

Par le lemme de Yoneda, cela montre que $\tilde\phi$ est scind\'ee, donc que $N$ est
\emph{facteur direct} de $h(C)$. En particulier, $N$ est \emph{effectif} s'il existe; dans ce cas, il définit donc aussi $f_\sharp\un$ dans $\Chow^\eff(k)$.

Soit $h(C)=h_0(C)\oplus h_1(C)\oplus h_2(C)$ une d\'ecomposition de Chow-K\"unneth
d\'etermin\'ee par un point rationnel de $C$. On sait que $\End(h(C))$ est une $\Q$-alg\`ebre
primaire, et que les idempotents d\'efinissant la d\'e\-com\-po\-si\-tion ci-dessus sont
centraux modulo le radical. En appliquant le lemme \ref{l10.1} b), on peut donc \'ecrire
$N=N_0\oplus N_1\oplus N_2$, o\`u $N_i$ est isomorphe \`a un facteur direct de $h_i(C)$.

Rappelons que $h_0(C)=\un$ et $h_2(C)=\bL$. Nous utiliserons les formules
\begin{align*}
\Chow(k)(\un,\bL)&=\Chow(k)(\bL,\un)=0\\
\Chow(k)(\un,h_1(A))&=\Chow(k)(h_1(A),\bL)=A(k)\otimes\Q\\
\Chow(k)(\bL,h_1(A))&=\Chow(k)(h_1(A),\un)=0
\end{align*}
si $A/k$ est une vari\'et\'e ab\'elienne.

Pour $X=\Spec k$, on obtient
\[\Chow(k)(N_0,\un)=\Chow(k)(N,\un)\iso \Q\]
ce qui montre que $N_0=\un$.

En appliquant l'adjonction \`a $h_1(B)$, o\`u $B$ est une vari\'et\'e
ab\'elienne, on obtient aussi
\[\Chow(k)(N,h_1(B)) = \Chow(K)(\un,h_1(B_K))=B(K)\otimes \Q\]
et donc
\[\Chow(k)(N_1,h_1(B))=\left(B(K)/B(k)\right)\otimes\Q =
\Chow(k)(h_1(C),h_1(B)).\] 

Comme cet isomorphisme est compatible \`a l'inclusion $N_1\inj h_1(C)$, cela implique
par Yoneda que $N_1 = h_1(C)$.

En appliquant enfin l'adjonction \`a $\bL\in \Chow(k)$, on obtient
\begin{multline*}
\Chow(k)(N_1,\bL)\oplus \Chow(k)(N_2,\bL)=\Chow(k)(N,\bL)\\ =
\Chow(K)(\un,\bL)=0
\end{multline*}
d'o\`u $N_2=0$. Cela donne \'egalement
\[\Chow(k)(N_1,\bL)=\Chow(k)(h_1(C),\bL)\allowbreak=\Pic^0(C)\otimes\Q=0\]
ce qui contredit $g>0$ si $k$ n'est pas la cl\^oture alg\'ebrique d'un corps fini (lemme \ref{l7.2}).\end{proof}

En revanche:

\enlargethispage*{20pt}

\begin{prop}\phantomsection\label{p7.1} Supposons que $k$ soit alg\'ebrique sur un corps fini.\\
a) Supposons vraie la conjecture de Beilinson (et Tate): $\Chow^\eff(k)\iso \Mot_\num^\eff(k)$. Alors 
 $f_\sharp \un$ existe pour toute extension de type fini $K/k$.\\
b) Si $\degtr(K/k)=1$, $f_\sharp\un$ existe inconditionnellement, et m\^eme dans $\Chow(k)$.
\end{prop}

\begin{proof} a)
Le foncteur $\Chow^\eff(k)\to \Chow^\o(k)$ a un adjoint \`a gauche
$i$ (celui qui existe sans hypoth\`ese sur $k$ au niveau des motifs num\'eriques, cf.
corollaire \ref{c2}). En testant sur
$h(X)$ pour $X/k$ projective lisse, on voit imm\'ediatement que le motif $if_\sharp^\o\un$ fait
l'affaire, o\`u $f_\sharp^\o$ est l'adjoint \`a gauche du th\'eor\`eme \ref{t4}.

b) L'\'enonc\'e signifie que le foncteur $\Chow(k)\to \Vec_\Q$ donn\'e par $N\mapsto \Chow(\un,N_K)$ est corepr\'esentable (sans supposer vraie la conjecture de Beilinson). Comme tout motif de $\Chow(k)$ est de la forme $M(-n)$ pour $M\in \Chow^\eff(k)$ et $n\in\N$, il suffit de montrer que $f_\sharp\bL^n$ est repr\'esentable dans $\Chow^\eff(k)$ pour tout $n\ge 0$. 

Un calcul analogue \`a celui
de la d\'emonstration de la proposition \ref{p10.1} (en testant sur les motifs du type
$\bL^n,h(A)(n)$ et $\bL^{n+1}$) montre que, n\'ecessairement, on a
$f_\sharp\bL^n=(f_\sharp\un)(n)=(\un\oplus h_1(C))(n)$. Il suffit de voir que,  pour toute vari\'et\'e projective lisse $X$ de dimension $d$, l'application naturelle
\begin{multline*}
\Chow^\eff(k)((\un\oplus h_1(C))(n),h(X))\\
\to  \Chow^\eff(k)(\bL^n,h(X_K)) = CH^{d-n}(X_K)\otimes \Q
\end{multline*}
est bijective.

Le membre de gauche est facteur direct de
\[\Chow^\eff(k)(h(C)(n),h(X))=CH^{d-n}(C\times X)_\Q.\]

La suite exacte de localisation
\[\bigoplus_{c\in C} CH^{d-n-1}(X_{k(c)})\to CH^{d-n}(C\times X)\to CH^{d-n}(X_K)\to 0\]
induit une suite exacte
\[\bigoplus_{[E:k]<\infty} CH^1(C_E)\otimes CH^{d-n-1}(X_E)\by{(\alpha_E)} CH^{d-n}(C\times X)\to CH^{d-n}(X_K)\to 0\]
o\`u $\alpha_E$ est le produit d'intersection suivi du transfert. En appliquant la d\'ecomposition motivique de $C$, on trouve une suite exacte
\begin{multline*}
\bigoplus_{[E:k]<\infty} \Pic^0(C_E)_\Q\otimes CH^{d-n-1}(X_E)_\Q\by{(\alpha_E)} \Chow^\eff(k)(h_1(C)(n),h(X))\\
\to CH^{d-n}(X_K)_\Q\to 0.
\end{multline*}

Mais $\Pic^0(C_E)$ est de torsion pour tout $E$.
\end{proof}

\enlargethispage*{60pt}

\section{Quelques calculs}\label{s8}

Dans cette section, \emph{tous les motifs sont à coefficients rationnels}. On abrège $\Chow(k,\Q)$ en $\Chow(k)$, etc.

\subsection{Motifs purs birationnels}

\begin{prop}\phantomsection\label{p8.1} Soit $C$ une courbe projective lisse, g\'e\-o\-m\'e\-tri\-que\-ment connexe sur un corps parfait $k$, 
et $K=k(C)$. Consid\'erons le foncteur $f_\sharp:\Chow^\o(K)\to \Chow^\o(k)$ du 
th\'eor\`eme \ref{t4}. Alors:\\
a) On a $f_\sharp \un=\un\oplus h_1^\o(C)$.\\
b) Soit $\Gamma$ une courbe projective lisse connexe sur $K$, et soit 
$S$ une surface projective lisse géométriquement connexe sur $k$, fibr\'ee sur $C$ de fibre g\'en\'erique $\Gamma$. Soit $J$ la jacobienne de $\Gamma$. Alors on a une suite exacte scind\'ee
\begin{equation}\label{eq11.0}
0\to t_2^\o(S)\to f_\sharp h_1^\o(\Gamma)\to h_1^\o(\IM_{K'/k} J)\to 0
\end{equation}
naturelle en $\Gamma$ (pour les correspondances birationnelles), o\`u $K'$ est le corps des constantes de $\Gamma$.
\end{prop}

\begin{rque}\phantomsection \label{r8.1} En utilisant la proposition \ref{p5.1} et  le corollaire \ref{c2}, on voit que pour le morphisme de théories motiviques $\Mot_\rat^\o\to \Mot_\num^\o$, l'effet du morphisme de changement de base sur $\un$ (resp. $h_1(\Gamma)$) est
\[\un\inj \un\oplus h_1^\o(C)\quad  (\text{resp. } h_1^\o(\IM_{K'/k} J)\inj h_1^\o(\IM_{K'/k} J)\oplus t_2^\o(S))\]
ce qui illustre spectaculairement la proposition \ref{p3.1} et fournit un scindage canonique de \eqref{eq11.0} dans $\Mot_\num^\o$. Ce dernier point n'est pas surprenant, $\Mot_\num^\o$ étant semi-simple et les termes de poids différents.
\end{rque}

\begin{proof} a) C'est clair, puisque $f_\sharp\un = f_\sharp h^\o(\Spec K) = h^\o(C)$ d'apr\`es 
la preuve du th\'eor\`eme \ref{t4}, et que $h^\o(C)=\un\oplus h_1^\o(C)$.

b) De m\^eme, on a 
\[f_\sharp h^\o(\Gamma)=h^\o(S)=\un\oplus h_1^\o(S)\oplus t_2^\o(S).\]

Pour clarifier les calculs qui suivent, rappelons la formule suivante \cite[Cor. 7.8.5 (ii)]{kmp}: si $X,Y\in \Sm^\proj(k)$ avec $\dim X,\dim Y\le 2$, alors
\begin{equation}\label{eq11.5}
\Chow^\o(k)(h_i^\o(X),h_j^\o(Y))=0\text{ si } i>j.
\end{equation}

Ceci montre que $h^\o(X)$ est muni d'une \emph{filtration canonique}
\begin{equation}\label{eq11.1}
0\subset h_{> 1}^\o(X)\subset h_{> 0}^\o(X)\subset h^\o(X)
\end{equation}
o\`u les ``inclusions'' sont des monomorphismes scind\'es, de suppl\'ementaires successifs
\[h^\o_2(X)=t_2^\o(X), h^\o_1(X), h^\o_0(X).\]

La filtration \eqref{eq11.1} est fonctorielle en $X$ (pour les correspondances birationnelles) et ne d\'epend pas du choix d'une d\'ecomposition de Chow-K\"unneth de $X$.

Appliquons maintenant $f_\sharp$ \`a cette filtration pour $h^\o(\Gamma)$, qu'on peut \'ecrire comme une suite exacte scind\'ee
\[0\to h_1^\o(\Gamma)\to h^\o(\Gamma)\to h_0^\o(\Gamma)\to 0.\]

\enlargethispage*{20pt}

On obtient ainsi une suite exacte scind\'ee
\[0\to f_\sharp h_1^\o(\Gamma)\to f_\sharp h^\o(\Gamma)\to f_\sharp h_0^\o(\Gamma)\to 0\]
soit
\begin{equation}\label{eq11.2}
0\to f_\sharp h_1^\o(\Gamma)\to h^\o(S)\to h^\o(C')\to 0
\end{equation}
o\`u $C'$ est la courbe projective lisse de mod\`ele $K'$. Comme $k(S)/k$ est r\'eguli\`ere, on a $h_0^\o(S)\iso h_0^\o(C') = \un$. Par fonctorialit\'e de la filtration \eqref{eq11.1}, on d\'eduit donc de \eqref{eq11.2} une suite exacte scind\'ee
\[0\to f_\sharp h_1^\o(\Gamma)\to h^\o_{>0}(S)\to h^\o_{>0}(C')\to 0
\]
qu'on peut ins\'erer dans un diagramme commutatif de suites exactes scind\'ees
\[\begin{CD}
&&&& 0\\
&&&& @VVV\\
&&&& t_2^\o(S)\\
&&&& @VVV\\
0@>>> f_\sharp h_1^\o(\Gamma)@>>> h^\o_{>0}(S)@>>> h^\o_{>0}(C')@>>> 0.\\
&&&&@VVV @V\wr VV\\
&&&&h_1^\o(S)@>>> h_1^\o(C')\\
&&&&@VVV \\
&&&&0
\end{CD}\]

D'apr\`es la suite exacte \`a isog\'enie pr\`es \cite[prop. 3.4]{hp}
\[0\to \Pic^0_{C'/k}\to \Pic^0_{S/k}\to \Tr_{K'/k} J\to 0\]
qu'on dualise, on a une suite exacte scind\'ee
\[
0\to h_1^\o(\IM_{K'/k}J)\to h_1^\o(S)\to h_1^\o(C')\to 0
\]
d'o\`u on d\'eduit imm\'ediatement la suite exacte \eqref{eq11.0}; sa construction montre qu'elle est fonctorielle pour l'action des correspondances de Chow (birationnelles).
\end{proof}

\subsection{Le motif de Tate-\v Safarevi\v c}

\begin{prop}\phantomsection\label{c8.1} Soit $A$ une $K$-vari\'et\'e ab\'elienne. Alors on a une suite exacte scind\'ee dans $\Chow^\o(k)$, naturelle en $A$ (pour les homomorphismes de variétés abéliennes)
\begin{equation}\label{eq11.3}
0\to \sha^\o(A,K/k)\to f_\sharp h_1^\o(A)\to h_1^\o(\IM_{K/k} A)\to 0
\end{equation}
o\`u $\sha^\o(A,K/k)$ est facteur direct de $t_2^\o(S)$ pour une surface $S$ convenable. 
\end{prop}

\begin{proof} Soit $\Gamma$ une courbe ample trac\'ee sur $A$, de sorte que le morphisme
\[J=\Alb(\Gamma)\to \Alb(A)=A\]
soit surjectif. On a donc un \'epimorphisme scind\'e
\[h_1^\o(J)\surj h_1^\o(A)\]
d'o\`u un \'epimorphisme scind\'e
\[f_\sharp h_1^\o(J)\surj f_\sharp h_1^\o(A).\]

Choisissons une surface $S$ comme dans la proposition \ref{p8.1} b). Si $A\to J$ est une section de $J\to A$ \`a isog\'enie pr\`es, de projecteur associ\'e $\pi\in \End^0(J)=\End_{\Chow^\o(K)}(h_1^\o(\Gamma))\subset \End_{\Chow^\o(K)}(h^\o(\Gamma))$, alors le projecteur $f_\sharp(\pi)$ op\`ere sur la suite exacte \eqref{eq11.0}, y d\'ecoupant une suite exacte scind\'ee
\[0\to f_\sharp(\pi) t_2^\o(S)\to  f_\sharp h_1^\o(A)\to f_\sharp(\pi) h_1^\o(\IM_{K/k} J)\to 0
\]
 de la forme \eqref{eq11.3}. Comme $\Hom(f_\sharp(\pi) t_2^\o(S), f_\sharp(\pi) h_1^\o(\IM_{K'/k} J))=0$ par \eqref{eq11.5}, cette filtration est unique; en particulier, elle est ind\'ependante du choix de $S$, puis de $\Gamma$, et est fonctorielle en $A$ (pour les homomorphismes de vari\'et\'es ab\'eliennes).

Il reste \`a identifier $f_\sharp(\pi) h_1^\o(\IM_{K/k} J)$ \`a $h_1^\o(\IM_{K/k} A)$: cela r\'esulte du fait que le projecteur $\pi$ op\`ere (\`a isog\'enie pr\`es) sur $\IM_{K/k} J$ et y d\'ecoupe $\IM_{K/k} A$, par fonctorialité de la $K/K$-image.
\end{proof}

Rappelons maintenant que le foncteur de projection $d_{\le n} \Chow^\eff(k)\to d_{\le n} \Chow^\o(k)$ admet un adjoint à droite $P_n^\sharp$ pour $n\le 2$ \cite[th. 7.8.8]{kmp}.\footnote{Attention! Cela devient faux pour $n\ge 3$, cf. \cite[th. 4.3.2 et 4.3.3]{birat-pure}.} Plus précisément, la même référence donne, avec les notations précédentes:
\[P^\sharp\un =\un,\quad P^\sharp h_1^\o(A) = h_1(A),\quad P^\sharp t_2^\o(S)=t_2(S).\]

On en déduit:

\begin{cor}\phantomsection\label{c8.2} Soit $A$ une $K$-vari\'et\'e ab\'elienne. Alors on a une suite exacte scind\'ee dans $\Chow^\eff(k)$, naturelle en $A$ (pour les homomorphismes de variétés abéliennes)
\[
0\to \sha(A,K/k)\to P^\sharp f_\sharp h_1^\o(A)\to h_1(\IM_{K/k} A)\to 0
\]
o\`u $\sha(A,K/k)$ est facteur direct de $t_2(S)$ pour une surface $S$ convenable. \qed
\end{cor}

\begin{defn}\phantomsection \label{d8.1} On appelle $\sha(A,K/k)$ le \emph{motif de Tate-\v Safarevi\v c} de $A$.
\end{defn}

Soit $H$ une cohomologie de Weil classique (\S \ref{s2.2.2}). Il est connu qu'alors $\dim H^1(X) = 2 \dim \Alb(X)$. Cette condition implique que $H^i(h_2(S))=0$ pour $i\ne 2$ \cite[th. 2A9]{kleiman-dix}. Il en est donc de même pour $\sha(A,K/k)$, qui est ainsi ``de poids $2$'' pour $H$.

Soit $l$ un nombre premier inversible dans $k$, et fixons une clôture séparable $k_s$ de $k$. En composant $\sha(-,K/k)$ avec le foncteur réalisation $l$-adique $\tilde \rho_l$ de la définition \ref{d2.1}, on obtient un foncteur contravariant
\[\Sha_l=\tilde \rho_l\circ \sha(-,K/k):\Ab(K,\Q)\to \Rep_{\Q_l}^*(G_k)\]
où $G_k=Gal(k_s/k)$. On a un autre tel foncteur, covariant:
\[\Sha'_l=V_l(\Sha(-,K k_s)):\Ab(K,\Q)\to \Rep_{\Q_l}^*(G_k)\]
où
\[\Sha(A,K k_s) = \Ker (H^1( k_s(C),A)\to \prod_{v\in C_{(0)}}  H^1( k_s(C)_v,A)) \]
est le \emph{groupe de Tate-\v Safarevi\v c géométrique} de $A\in \Ab(K)$, muni de l'action canonique de $G_k$. Comme remarqué dans \cite[(4.9)]{BrIII}, on peut remplacer ci-dessus les complétés $ k_s(C)_v$ par les hensélisés correspondants, et le produit par une somme directe.

Le théorème suivant justifie la terminologie de la définition \ref{d8.1}:

\begin{thm}\label{t8.1} On a un isomorphisme dans $\Rep_{\Q_l}^*(G_k)$:
\[u_A:\Sha_l(A)\simeq \Sha'_l(A^*)(-1).\]
naturel en $A\in \Ab(K,\Q)$.
\end{thm}

La démonstration dépasserait le cadre de cet article: elle sera donnée ailleurs. Contentons-nous ici de rendre l'énoncé plausible en traitant le cas particulier $A=J(\Gamma)$, où $\Gamma$ est une $K$-courbe projective lisse géométriquement connexe. Soit $S$ un modèle projectif lisse de $\Gamma$ sur $k$, fibré sur $C$ par un morphisme plat $f$. Les calculs de \cite[\S 4]{BrIII}, utilisant la suite spectrale de Leray pour $f$, donnent alors un isomorphisme
\begin{equation}\label{eq8.1}
 \Br(S_{ k_s})\iso \Sha(J(\Gamma),K k_s) 
\end{equation}
dans le quotient de la catégorie des $G_k$-modules par la sous-catégorie de Serre formée des $G_k$-modules finis. Cela fournit $u_A$ dans ce cas, compte tenu de l'isomorphisme 
\[\tilde \rho_l(t_2(S)) \simeq \Coker(\NS(S_{ k_s})\otimes \Q_l(-1)\by{\cl} H^2(S_{ k_s},\Q_l)\simeq  V_l(\Br(S_{ k_s}))(-1)\] 
\cite[Prop. 7.2.3]{kmp}, et de l'autodualité de $J(\Gamma)$. 

Le point délicat est de montrer que $u_A$ s'étend en un isomorphisme naturel sur $\Ab(K,\Q)$. Pour pouvoir procéder comme dans la preuve de la proposition \ref{c8.1}, il faut au moins savoir que \eqref{eq8.1} est compatible à l'action des endomorphismes de $J(\Gamma)$ via le foncteur  $P^\sharp f_\sharp$ intervenant dans le corollaire \ref{c8.2}. 
 Pour cela, il faut utiliser la catégorie des ``motifs de Chow sur $C$'' définie par Corti et Hanamura \cite{CH}, et leur foncteur de réalisation $l$-adique.

\section{Motifs au point g\'en\'erique}\label{s9}

\begin{defn}\phantomsection Soit $(A,\sim)$ un couple ad\'equat, et soit $k$ un corps. Pour tout $n\ge 0$, 
on note $d_{\le n} \Mot_\sim^\eff(k,A)$ (resp. $d_{\le n} \Mot_\sim^\o(k,A)$) la sous-cat\'egorie 
\'epaisse de $\Mot_\sim^\eff(k,A)$ (resp. de $\Mot_\sim^\o(k,A)$) engendr\'ee par les motifs des 
vari\'et\'es de dimension $\le n$. On note $d_n \Mot_\sim^\eff(k,A)$ (resp. $d_n \Mot_\sim^\o(k,A)$) 
le quotient de $d_{\le n} \Mot_\sim^\eff(k,A)$ (resp. de $d_{\le n} \Mot_\sim^\o(k,A)$) par l'id\'eal 
des morphismes se factorisant \`a travers un objet de $d_{\le n-1} \Mot_\sim^\eff(k,A)$ 
(resp. de $d_{\le n-1} \Mot_\sim^\o(k,A)$).
\end{defn}

Ces catégories ont été étudiées dans \cite[\S 7.8.3]{kmp}: on a pour tout $n\ge 0$ un diagramme commutatif de cat\'egories et foncteurs:
\begin{equation}\label{eq9.2}
\xymatrix{
& d_{\le n} \Mot_\sim^\eff(k,A)\ar[dl]_{R_n}\ar[dr]^{P_n}\\
d_n \Mot_\sim^\eff(k,A)\ar[dr]^{Q_n} && d_{\le n} \Mot_\sim^\o(k,A)\ar[dl]_{S_n}\\
& d_n \Mot_\sim^\o(k,A)
}
\end{equation}

Rappelons de \cite{kmp} que la ``dualit\'e de Cartier sup\'erieure''
$M\mapsto \uHom(M,\bL^{\otimes n})$ induit une dualit\'e sur $d_n \Mot_\sim^\o(k,A)$. De plus, si 
$\sim =\rat$ et $A$ est une $\Q$-algèbre, on a pour deux vari\'et\'es projectives lisses $X,Y$ de dimension $n$
\begin{equation}\label{eq9.3}
d_n \Chow^\o(k,A)(\bar h^\o(X),\bar h^\o(Y)) =CH^n(X\times Y)\otimes A/\equiv
\end{equation}
où $\equiv$ est le sous-$A$-module engendr\'e par les classes de sous-vari\'et\'es irr\'eductibles 
qui ne sont pas dominantes sur l'un des facteurs $X,Y$ (\cite[(7.19)]{kmp}; maintenant $\Q$ pourrait être remplacé par $\Z[1/p]$). Autrement dit, 
$d_n \Mot_\sim^\o(k,A)(\bar h^\o(X),\bar h^\o(Y))$ est le $A$-module des \emph{correspondances 
au point g\'en\'erique} de Bloch, Rovinsky et Beilinson \cite{beilinson}. Ces groupes de correspondances apparaissent aussi dans Fulton \cite[ex. 16.1.2 (b)]{fulton}.

Voici une autre expression de ce groupe de correspondances en termes de $0$-cycles:
\enlargethispage*{20pt}

\begin{lemme}\phantomsection\label{l9.1} On a
\begin{multline*}
d_n \Chow^\o(k,A)(\bar h^\o(X),\bar h^\o(Y)) =\\
\Coker\left(\bigoplus_{Z\subsetneq Y} CH_0(Z_{k(X)})\to CH_0(Y_{k(X)})\right)\otimes A
\end{multline*}
où $Z$ décrit les fermés propres de $Y$.
\end{lemme}

\begin{proof} C'est clair en écrivant l'idéal $\equiv$ comme somme de deux idéaux: l'un ``à droite'' et l'autre ``à gauche''.
\end{proof}

\begin{thm}\phantomsection\label{t9.1} Soit $f:k\to K$ une extension de type fini, de degr\'e de transcendance $d$. Soit $A$ une $\Z[1/p]$-algèbre, où $p$ est l'exposant caractéristique. Alors, pour tout $n\ge 0$,\\
a)  $f_\sharp d_{\le n} \Mot_\sim^\o(K,A)\subset d_{\le n+d} \Mot_\sim^\o(K,A)$.\\
b) $f_\sharp$ induit un foncteur $f_\sharp^n: d_n\Mot_\sim^\o(K,A)\to d_{n+d} \Mot_\sim^\o(K,A)$.\\
c) Soient $X,Y\in \Sm^\proj(K)$ de dimension $n$. Supposons que $K/k$ admette un mod\`ele projectif 
lisse $S$ et que $X$ et $Y$ se prolongent en des $S$-sch\'emas projectifs $\sX,\sY$, lisses sur $k$. 
Alors le conoyau de
\[
d_n \Chow^\o(K,A)(\bar h^\o(X),\bar h^\o(Y))\to d_{n+d} 
\Chow^\o(k,A)(f_\sharp^n \bar h^\o(X),f_\sharp^n \bar h^\o(Y))
\]
est isomorphe à
\[ \frac{CH^{n+d}(\sX\times_k\sY-\sX\times_S\sY)\otimes A}{\overline{\equiv}}
\]
o\`u $\overline{\equiv}$ d\'esigne l'image de $\equiv\subset CH^{n+d}(\sX\times_k\sY)\otimes A$. Il est aussi isomorphe à
\[\Coker\left(\bigoplus_{Z\subsetneq \sY} CH_0(Z_{k(\sX)})\to CH_0(\sY_{k(X)}-Y_{K(X)})\right)\otimes A\]
(noter que $k(\sX)=K(X)$).
\end{thm}

\begin{proof} a) est clair vu la preuve du th\'eor\`eme \ref{t4}, et b) r\'esulte imm\'ediatement de 
a). D\'emontrons c): pour commencer, l'homomorphisme
\begin{align}
\Chow^\o(K,A)(h^\o(X),h^\o(Y))&\by{f_\sharp}  
\Chow^\o(k,A)(f_\sharp h^\o(X),f_\sharp h^\o(Y))\label{eq14.1}\\
\intertext{est induit par l'unit\'e}
h^\o(Y)&\to f^*f_\sharp h^\o(Y)\notag\\
\intertext{elle-m\^eme d\'ecrite par l'immersion ferm\'ee}
Y&\to \sY_K\label{eq14.2}\\
\intertext{fibre g\'en\'erique du morphisme graphe}
\sY&\to\sY\times_k S.\notag
\end{align}

\enlargethispage*{25pt}

L'homomorphisme \eqref{eq14.1} est donc donn\'e par
\[CH_0(Y_{K(X)})\otimes A\to CH_0(\sY_{k(\sX)})\otimes A\]
o\`u le $K(X)$-morphisme $Y_{K(X)})\to \sY_{k(\sX)}$ est d\'eduit du $K$-morphisme 
\eqref{eq14.2} par extension des scalaires de $K$ \`a $K(X)$. Ce $K(X)$-morphisme est la fibre 
g\'en\'erique de l'immersion ferm\'ee
\[\sY\times_S\sX\to \sY\times_k\sX.\]

On a un diagramme commutatif
\[\begin{CD}
CH_n(\sX\times_S\sY) @>>> CH_n(\sX\times_k\sY) &\to& CH_n(\sX\times_k\sY-\sX\times_S\sY) &\to& 0\\
@VVV @VVV @VVV\\ 
CH_0(Y_{K(X)})@>>> CH_0(\sY_{k(\sX)})&\to& CH_0(\sY_{k(X)}-Y_{K(X)}) &\to& 0\\
@VVV @VVV\\
\displaystyle\frac{CH_n(X\times_K Y)}{\equiv}@>>> \displaystyle\frac{CH_{n+d}(\sX\times_k\sY)}{\equiv}
\end{CD}\]
o\`u la suite sup\'erieure est exacte et o\`u les fl\`eches verticales sont surjectives. D'o\`u le premier isomorphisme de c); le second se déduit du lemme \ref{l9.1}.
\end{proof}

\begin{rque}\phantomsection\label{r9.2} Pour $d=n=1$, on obtient la ``description plus fonctorielle'' de l'homomorphisme $r$ de \cite[Prop. 3.1 a)]{surfsurcourbe}, promise dans la preuve de loc. cit. Ceci permet de montrer que le motif $\sha(A,K/k)$ de la définition \ref{d8.1} coïncide avec celui introduit dans \cite[prop. 3.2]{surfsurcourbe} (ce qui démontre l'unicité de ce dernier à isomorphisme unique près, et sa fonctorialité en $A$). En effet, en appliquant le foncteur $S_2$ de \eqref{eq9.2} à \eqref{eq11.3} et en tenant compte du théorème \ref{t9.1} b), on obtient
\[f_\sharp^2(S_1(h_1^\o(A))\simeq S_2(\sha^\o(A,K/k)).\]

Revenant à la démonstration de la proposition \ref{c8.1}, on voit avec ses notations que $S_2(\sha^\o(A,K/k))$ est découpé sur $S_2(t_2^\o(S))$ par le projecteur $f_\sharp^2(S_1(\pi))$ qui, vu \eqref{eq9.3}, n'est autre que celui de \cite[dém. de la prop. 3.2]{surfsurcourbe}. On conclut en observant que 
\[\End_{\Mot_\rat^\eff(k,\Q)}(t_2(S))\iso \End_{d_2\Mot_\rat^\o(k,\Q)}(S_2(t_2^\o(S)))\]
via $S_2P_2$, d'après \cite[th. 7.4.3]{kmp}.
\end{rque}


\begin{rque}\phantomsection\label{r9.1} On peut comparer équivalences rationnelle et numérique. Avec les notations du théorème \ref{t9.1} et $A=\Q$ (qu'on omet dorénavant), cela donne un diagramme commutatif:
\begin{equation}\label{eq9.1}\begin{CD}
d_n \Chow^\o(K)(\bar h^\o(X),\bar h^\o(Y))@>f_\sharp^n>> d_{n+d} \Chow^\o(k)(\bar h^\o(\sX),\bar h^\o(\sY))\\
@V\pi_K VV @V\pi_k VV\\
d_n \Mot_\num^\o(K)(\bar h^\o(X),\bar h^\o(Y))@>f_\sharp^n>> d_{n+d} \Mot_\num^\o(k)(\bar h^\o(\sX),\bar h^\o(\sY)).
\end{CD}
\end{equation}

Prenons $d=n=1$ et $X=Y$; donc $X$ est une courbe sur $K$ et $\sX$ est une surface sur $k$, modèle projectif lisse de $X$. Dans \eqref{eq9.1}, le premier membre est une $\Q$-algèbre semi-simple d'après le théorème de Weil \cite[ch. VI, th. 22]{weil} et la théorie des variétés abéliennes (complète réductibilité de Poincaré); plus précisément, $\pi_K$ est un isomorphisme. Les morphismes horizontaux sont des homomorphisme de $\Q$-algèbres, puisqu'il proviennent de foncteurs. Si celui du haut est surjectif, alors son second membre est aussi semi-simple (en particulier de dimension finie); mais alors $\pi_k$ est un homomorphisme surjectif de $\Q$-algèbres semi-simples, ce qui nous rapprocherait fortement de la conjecture de Bloch pour $\sX$ sous la forme reformulée par ce dernier, Rovinsky et Beilinson \cite[1.4]{beilinson}.

\enlargethispage*{40pt}

Partant de $\sX$, cette stratégie est sans espoir si on choisit la fibration $\sX\to S$ au hasard. Prenons par exemple la projection $\pi:\sX=E_1\times E_2\to E_1=S$, où $E_1$ et $E_2$ sont deux courbes elliptiques non isogènes. On calcule facilement que $t_2(E_1\times E_2) = h_1(E_1)\otimes h_1(E_2)$. En tant que motif numérique, son algèbre d'endomorphismes est donc $\End^0(E_1)\otimes \End^0(E_2)$; en particulier, cette algèbre est toujours un corps et donc $t_2(E_1\times E_2)$ est simple. D'autre part, l'anneau des endomorphismes de $h_1((E_2)_{k(E_1})$ (fibre générique de $\pi$) est $\End^0((E_2)_{k(E_1)}=\End^0(E_2)$, et on calcule facilement que l'homomorphisme $\End^0(E_2)\to \End^0(E_1)\otimes \End^0(E_2)$ induit par $f_\sharp^2$ est donné par $f\mapsto 1\otimes f$. Si $E_1$ est à multiplication complexe, il n'est pa surjectif; a fortiori, le morphisme horizontal supérieur de \eqref{eq9.1} ne peut pas être surjectif.

\`A l'autre extrême, si l'on fibre $\sX$ sur $\P^1$ à l'aide d'un pinceau de Lefschetz (à un éclatement de $\sX$ près), la jacobienne de la fibre générique est produit à isogénie près de sa partie constante, $\Pic^0_{\sX/k}$, et d'une variété abélienne dont le corps des endomorphismes est $\Q$ pour des raisons de monodromie: je remercie Claire Voisin de m'avoir fait observer ce point. Dans l'exemple $\sX=E_1\times E_2$, le morphisme horizontal inférieur de \eqref{eq9.1} n'est encore pas surjectif. Trouver une bonne fibration telle que la jacobienne de la fibre générique ait un corps d'endomorphismes assez gros semble être une question difficile. 
\end{rque}


\begin{thebibliography}{SGA4-I}
\bibitem{nrsm} Y. Andr\'e, B. Kahn {\it Nilpotence, radicaux et structures mono\"\i dales}
(avec un appendice de P. O'Sullivan), Rend. Sem. Mat. Univ. Padova {\bf 108} (2002), 107--291.
\bibitem{aknote} Y. Andr\'e, B. Kahn {\it Construction inconditionnelle de groupes de Galois
motiviques}, C. R. Acad. Sci. Paris {\bf 334} (2002), 989--994.
\bibitem{ayoub} J. Ayoub Les six op\'erations de Grothendieck et le formalisme des cycles
\'evanescents dans le monde motivique, I,  Ast\'erisque  {\bf 314}  (2007) [2008].
\bibitem{beilinson} A. Beilinson {\it Remarks on $n$-motives and correspondences at the 
generic point}, {\it in} Motives, polylogarithms and Hodge theory {\bf I} (Irvine, CA, 1998), 
Int. Press Lect. Ser., {\bf 3, I}, Int. Press, 2002, 35--46.  
\bibitem{cis-deg} D.-C. Cisinski, F. D\'eglise {\it Triangulated categories of motives}, pr\'epublication, 2012.
\bibitem{CH} A. Corti, M. Hanamura {\it Motivic decomposition and intersection Chow groups}, I. Duke Math. J. {\bf 103} (2000), 459--522.
\bibitem{ct-s} J.-L. Colliot-Th\'el\`ene, J.-J. Sansuc {\it La
$R$-\'equivalence sur les tores}, Ann. Sci. \'Ec. Norm. Sup. 
{\bf 10} (1977), 175--229.
\bibitem{debarre} O. Debarre Higher dimensional algebraic geometry, Springer, 2001.
\bibitem{deglise} F. D\'eglise {\it Finite correspondences and transfers over a regular base}, {\it in} Algebraic cycles and motives (J. Nagel, C. Peters, eds.), {\bf 1}, LMS Lect. Notes Series {\bf 343}, 138--205.
\bibitem{fulton} W. Fulton Intersection theory, Springer, 1984.
\bibitem{ghs} T. Graber, J. Harris, J. Starr {\it Families of rationally connected varieties}, J. Amer. Math. Soc. {\bf 16} (2003),  57--67.
\bibitem{BrIII} A. Grothendieck {\it Le groupe de Brauer, III: exemples et compl\'ements},  {\it in} Dix expos\'es sur la cohomologie des sch\'emas, Masson--North Holland, 1970.
\bibitem{resdu} R. Hartshorne Residues and Duality, Lect. Notes in
Math. {\bf 20}, Springer, 1966.
\bibitem{hp} M. Hindry, A. Pacheco {\it Sur le rang des jacobiennes sur les corps de fonctions}, 
Bull. SMF {\bf 133} (2005), 275--295.
\bibitem{ivorra} F. Ivorra {\it Réalisation $l$-adique des motifs triangulés géométriques, I}, Doc. Math. {\bf 12} (2007) 607--671.
\bibitem{dJ} A.J. de Jong {\it Smoothness, semi-stability and
alterations},  Publ. Math. IH\'ES {\bf 83} (1996), 51--93.
\bibitem{jannsen} U. Jannsen {\it Motives, numerical equivalence
and semi-simplicity}, Invent. Math. {\bf 107} (1992), 447--452.
\bibitem{jannsen2} U. Jannsen {\it Motivic sheaves and filtrations on Chow groups}, {\it in} 
Motives, Proc. Symp. pure Math. {\bf 55} (I), 245--302.
\bibitem{qjpam} B. Kahn {\it Zeta functions and motives}, Pure Appl. Math. Quarterly {\bf 5}
(2009), 507--570 [2008].
\bibitem{surfsurcourbe} B. Kahn {\it A motivic formula for the $L$-function of an abelian variety over a function field}, prépublication, 2014, \url{http://arxiv.org/abs/1401.6847}.
\bibitem{zetaL} B. Kahn Fonctions zêta et $L$ de variétés et de motifs, \url{https://arxiv.org/abs/1512.09250}.
\bibitem{kmp} B. Kahn, J.-P. Murre, C. Pedrini {\it On the transcendental part of the motive
of a surface}, {\it in} Algebraic cycles and motives, LMS Series {\bf 344} (2), Cambridge
University Press, 2007, 143-202.
\bibitem{Birat} B. Kahn, R. Sujatha {\it Birational motives, I (preliminary version)},
pr\'epublication, 2002,  \url{http://www.math.uiuc.edu/K-theory/0596/}.
\bibitem{localisation} B. Kahn, R. Sujatha {\it A few localisation theorems},
Homology, Homotopy Appl. {\bf 9} (2007), 137--161.
\bibitem{birat} B. Kahn, R. Sujatha {\it Birational geometry and localisation of
categories}, Doc. Math. -- Extra Volume Merkurjev (2015), 277--334.
\bibitem{birat-pure} B. Kahn, R. Sujatha {\it Birational motives, I: pure birational motives},  Annals of $K$-theory {\bf 1} (2016), 379--440.
\bibitem{birat-tri} B. Kahn, R. Sujatha {\it Birational motives, II: triangulated birational motives}, IMRN {\bf 2016}, doi: 10.1093/imrn/rnw184.
\bibitem{kleiman-dix} S. Kleiman {\it Algebraic cycles and the Weil conjectures}, {\it in} Dix expos\'es sur la cohomologie des sch\'emas (A. Grothendieck, N.H. Kuiper \'eds.), North Holland--Masson, 1968, 359--386.
\bibitem{kleiman} S. Kleiman {\it Equivalence relations between algebraic cycles}, Actes ICM 1970, Nice, I, 445--449.
\bibitem{kollar} J. Koll\'ar Rational curves on algebraic varieties,  Springer, 1996.
\bibitem{mvw} C. Mazza, V. Voevodsky, C. Weibel Lecture notes on Motivic
  cohomology, Clay Math. Monographs {\bf 2}, AMS, 2006.
\bibitem{morel} F. Morel {\it On the motivic $\pi\sb 0$ of the sphere spectrum}, {\it in} 
Axiomatic, enriched and motivic homotopy theory,  NATO Sci. Ser. II Math. Phys.
Chem. {\bf 131}, Kluwer, 2004, 219--260.
\bibitem{mv} F. Morel, V. Voevodsky {\it $\A^1$-homotopy theory of schemes}, Publ. Math. IH\'ES  {\bf  90} (1999), 45--143 (2001). 
\bibitem{neeman} A. Neeman Triangulated categories, Ann. of Math. Studies {\bf 148}, Princeton 
Univ. Press, 2001.
\bibitem{milne2} J.S. Milne Etale cohomology, Princeton Univ. Press, 1980.
\bibitem{qss} H.-G. Quebbemann, W. Scharlau, M. Schulte {\it Quadratic and Hermitian forms in additive and abelian categories},  J. Algebra {\bf 59} (1979), 264--289. 
\bibitem{ram} N. Ramachandran {\it Duality of Albanese and Picard $1$-motives}, $K$-Theory {\bf 22} (2001), 271--301.
\bibitem{roy-wald} D. Roy, M. Waldschmidt {\it Autour du th\'eor\`eme du sous-groupe alg\'ebrique},
Canad. Math. Bull. {\bf 36}  (1993), 358--367.
\bibitem{seb2} R. Sebastian {\it Examples of smash nilpotent cycles on rationally connected varieties}, J. of Algebra {\bf 438} (2015), 119--129.
\bibitem{suslin} A. Suslin {\it Motivic complexes over nonperfect fields}, Ann. of $K$-theory {\bf 2} (2017), 277--302.
\bibitem{vial} C. Vial {\it Chow-K\"unneth decomposition for 3- and 4-folds fibred by varieties with trivial Chow group of zero cycles}, J. Algebraic Geometry {\bf 24} (2015), 51--80.
\bibitem{vICM} V. Voevodsky {\it $\A^1$-homotopy theory}, Proc. ICM (Berlin, 1998). Doc. Math. 1998, Extra Vol. I, 579--604. 
\bibitem{voetri} V. Voevodsky {\it Triangulated categories of motives
over a field}, {\it in} E. Friedlander, A. Suslin, V. Voevodsky Cycles,
transfers and motivic cohomology theories, Ann. Math. Studies {\bf 143},
Princeton University Press, 2000, 188--238. 
\bibitem{weil} A. Weil Variétés abéliennes et courbes algébriques, Hermann, 1948.
\bibitem[SGA1]{SGA1} A. Grothendieck,  Revêtements étales et groupe fondamental (SGA1), nouvelle édition, Doc. Math. {\bf 3}, SMF, 2003.
\bibitem[SGA4-I]{SGA4} E. Artin, A. Grothendieck, J.-L. Verdier Th\'eorie
des topos et cohomologie \'etale des sch\'emas (SGA4), Vol. 1, Lect. Notes
in Math. {\bf 269}, Springer, 1972.
\end{thebibliography}
\end{document}